\documentclass[12pt,a4paper,english]{article}
\usepackage[T1]{fontenc}
\usepackage[latin9]{inputenc}
\usepackage{amsmath}
\usepackage{amssymb}

\makeatletter

\providecommand{\tabularnewline}{\\}

\usepackage{amsthm}\usepackage{fullpage}\usepackage{makeidx}\usepackage{amsfonts}\usepackage{latexsym}\makeindex

\def\C{{\mathbb{C}}}
\def\Z{{\mathbb{Z}}}

\theoremstyle{definition}
\newtheorem{lemma}{Lemma}[section]
\newtheorem{theorem}[lemma]{Theorem}

\newtheorem{proposition}[lemma]{Proposition}
\newtheorem{definition}[lemma]{Definition}

\newtheorem{remark}[lemma]{Remark}

\usepackage{babel}

\makeatother

\usepackage{babel}

\usepackage[blocks]{authblk}% The option is for block layout
\title{Fusion rules for the vertex operator algebra
$V_{L_{2}}^{A_{4}}$}
\author{Chongying Dong\thanks{Supported by NSF grants}}
\affil{Department of Mathematics, University of
California, Santa Cruz, CA 95064 USA}
\author{Cuipo(Cuibo) Jiang\thanks{Supported  by China NSF grants
10931006,11371245, the RFDP of China(20100073110052), and the
Innovation Program of Shanghai Municipal
Education Commission (11ZZ18)}}
\author{Qifen Jiang\thanks{Supported by China NSF grant 11101269 and a special fund of SJTU YG2011MS35} }
\affil{Department of Mathematics, Shanghai Jiaotong University, Shanghai 200240 China}
\author{Xiangyu Jiao}
\affil{Department of Mathematics, East China Normal University, Shanghai, 200241, China}
\author{Nina Yu}
\affil{Department of Mathematics, University of California, Riverside, CA 92521 USA}
\begin{document}
\maketitle

\begin{abstract}
The fusion rules for vertex operator algebra $V_{L_{2}}^{A_{4}}$ are determined.
\end{abstract}

\section{Introduction }
\def\theequation{1.\arabic{equation}}
\setcounter{equation}{0}

The remaining problem in the classification of rational vertex operator
algebras with central charge $c=1$ is the characterization of $V_{\mathbb{Z}\alpha}^{G}$,
where $\mbox{(\ensuremath{\alpha},\ensuremath{\alpha})=2},$ $G$
is a subgroup of $SO(3)$ isomorphic to $A_{4},$ $S_{4},$  $A_{5}$.
The vertex operator algebra $V_{\mathbb{Z}\alpha}^{G}$ has not been
understood fully since $G$ is not an abelian group. In the case
$G=A_{4}$, the rationality, $C_{2}$-cofiniteness and classification
of irreducible modules of $V_{\mathbb{Z}\alpha}^{A_{4}}$ have been
established in \cite{DJ4}. In this
paper, we determine the fusion rules for $V_{\mathbb{Z}\alpha}^{A_{4}}.$
A characterization of  $V_{\mathbb{Z}\alpha}^{A_{4}}$ is given in \cite{DJ5}.

One important tool in the determination of fusion rules is the quantum dimension of a module over a vertex operator algebra which has been studied systematically in \cite{DJX}. For a rational, $C_{2}$-cofinite,
self-dual vertex operator algebra of CFT type, quantum dimensions
of its irreducible modules have nice properties. In particular, the product of quantum dimensions of two
modules is equal to the quantum dimension of the fusion product of the modules. It turns out that this is
very helpful in determining fusion rules. It has been proved in \cite{DJ4} that the vertex operator
is rational, $C_{2}$-cofinite, self-dual vertex operator algebra of CFT type. So we can apply the results in \cite{DJX} on quantum dimensions to the vertex operator algebra $V_{\mathbb{Z}\alpha}^{A_{4}}.$
The fusion rules for the most cases can be determined by using the quantum dimensions.  For some fusion rules involving irreducible $V_{\mathbb{Z}\alpha}^{A_{4}}$-modules occurring in some twisted sectors,
 we need to find out the corresponding
$S$-matrix and use Verlinde formula to determine the remaining fusion rules.

The paper is organized as follows: In Section 2, we give some basic
definitions. In Section 3, we recall the vertex operator algebra $V_{\mathbb{Z}\alpha}^{A_{4}}$
and give the realization of all irreducible modules of $V_{\mathbb{Z}\alpha}^{A_{4}}$.
We compute the quantum dimensions of the irreducible $V_{\mathbb{Z}\alpha}^{A_{4}}$-modules in Section
4. The portion of
$S$-matrix that we need is listed in the appendix.

\section{Basics}
\def\theequation{2.\arabic{equation}}
\setcounter{equation}{0}

Let $\left(V,Y,1,\omega\right)$ be a vertex operator algebra (see
\cite{FLM}) and $g$ an automorphism of $V$ of finite order $T$.
Denote the decomposition of $V$ into eigenspaces of $g$ as:

\[
V=\oplus_{r\in\mathbb{Z}/T\text{\ensuremath{\mathbb{Z}} }}V^{r}
\]
where $V^{r}=\left\{ v\in V|gv=e^{2\pi ir/T}v\right\} $. Now we recall
notions of twisted modules for vertex operator algebras. Let $W\left\{ z\right\} $
denote the space of $W$-valued formal series in arbitrary complex
powers of $z$ for a vector space $W$.

\begin{definition}A \emph{weak $g$-twisted $V$-module} $M$ is
a vector space with a linear map
\[
Y_{M}:V\to\left(\text{End}M\right)\{z\}
\]

\[
v\mapsto Y_{M}\left(v,z\right)=\sum_{n\in\mathbb{Q}}v_{n}z^{-n-1}\ \left(v_{n}\in\mbox{End}M\right)
\]
which satisfies the following: for all $0\le r\le T-1$, $u\in V^{r}$,
$v\in V$, $w\in M$,

\[
Y_{M}\left(u,z\right)=\sum_{n\in\frac{r}{T}+\mathbb{Z}}u_{n}z^{-n-1},
\]

\[
u_{l}w=0\ for\ l\gg0,
\]

\[
Y_{M}\left(\mathbf{1},z\right)=Id_{M},
\]

\[
z_{0}^{-1}\text{\ensuremath{\delta}}\left(\frac{z_{1}-z_{2}}{z_{0}}\right)Y_{M}\left(u,z_{1}\right)Y_{M}\left(v,z_{2}\right)-z_{0}^{-1}\delta\left(\frac{z_{2}-z_{1}}{-z_{0}}\right)Y_{M}\left(v,z_{2}\right)Y_{M}\left(u,z_{1}\right)
\]

\[
z_{2}^{-1}\left(\frac{z_{1}-z_{0}}{z_{2}}\right)^{-r/T}\delta\left(\frac{z_{1}-z_{0}}{z_{2}}\right)Y_{M}\left(Y\left(u,z_{0}\right)v,z_{2}\right),
\]
 where $\delta\left(z\right)=\sum_{n\in\mathbb{Z}}z^{n}$. \end{definition}

\begin{definition}

A $g$-\emph{twisted $V$-module} is a weak \emph{$g$-}twisted $V$-module\emph{
}$M$ which carries a $\mathbb{C}$-grading induced by the spectrum
of $L(0)$ where $L(0)$ is the component operator of $Y(\omega,z)=\sum_{n\in\mathbb{Z}}L(n)z^{-n-2}.$
That is, we have $M=\bigoplus_{\lambda\in\mathbb{C}}M_{\lambda},$
where $M_{\lambda}=\{w\in M|L(0)w=\lambda w\}$. Moreover we require
that $\dim M_{\lambda}$ is finite and for fixed $\lambda,$ $M_{\frac{n}{T}+\lambda}=0$
for all small enough integers $n.$

\end{definition}

\begin{definition}An \emph{admissible $g$-twisted $V$-module} $M=\oplus_{n\in\frac{1}{T}\mathbb{Z}_{+}}M\left(n\right)$
is a $\frac{1}{T}\mathbb{Z}_{+}$-graded weak $g$-twisted module
such that $u_{m}M\left(n\right)\subset M\left(\mbox{wt}u-m-1+n\right)$
for homogeneous $u\in V$ and $m,n\in\frac{1}{T}\mathbb{Z}.$ $ $

\end{definition}

If $g=Id_{V}$ we have the notions of weak, ordinary and admissible
$V$-modules \cite{DLM3}.

\begin{definition}A vertex operator algebra $V$ is called \emph{$g$-rational}
if the admissible $g$-twisted module category is semisimple. $V$
is called \emph{rational} if $V$ is $1$-rational. \end{definition}

The following lemma about $g$-rational vertex operator algebras is
well known \cite{DLM3}.

\begin{lemma} If $V$ is $g$-rational and $M$ is an irreducible
admissible $g$-twisted $V$-module, then

(1) $M$ is a $g$-twisted $V$-module and there exists a number $\lambda\in\mathbb{C}$
such that $M=\oplus_{n\in\frac{1}{T}\mathbb{Z_{+}}}M_{\lambda+n}$
where $M_{\lambda}\neq0.$ $\lambda$ is called the conformal weight
of $M;$

(2) There are only finitely many irreducible admissible $g$-twisted
$V$-modules up to isomorphism. \end{lemma}

\begin{definition} We say that a vertex operator algebra $V$ is
\emph{$C_{2}$-cofinite} if $V/C_{2}(V)$ is finite dimensional, where
$C_{2}(V)=\langle v_{-2}u|v,u\in V\rangle.$ \end{definition}

\begin{remark} If $V$ is a vertex operator algebra satisfying $C_{2}$-cofinite
property, $V$ has only finitely many irreducible admissible modules
up to isomorphism \cite{DLM3, L2}.

\end{remark}

\begin{definition} Let $M=\bigoplus_{n\in\frac{1}{T}\mathbb{Z}_{+}}M(n)$
be an admissible $g$-twisted $V$-module, the\emph{ contragredient
module }$M'$ is defined as follows:
\[
M'=\bigoplus_{n\in\frac{1}{T}\mathbb{Z}_{+}}M(n)^{*},
\]
where $M(n)^{*}=\mbox{Hom}_{\mathbb{C}}(M(n),\mathbb{C}).$ The vertex
operator $Y_{M'}(v,z)$ is defined for $v\in V$ via
\begin{eqnarray*}
\langle Y_{M'}(v,z)f,u\rangle=\langle f,Y_{M}(e^{zL(1)}(-z^{-2})^{L(0)}v,z^{-1})u\rangle,
\end{eqnarray*}
where $\langle f,w\rangle=f(w)$ is the natural paring $M'\times M\to\mathbb{C}.$
\end{definition}

\begin{remark} 1.  $(M',Y_{M'})$ is an admissible $g^{-1}$-twisted
$V$-module \cite{FHL}.

2. We can also define the contragredient module $M'$ for a $g$-twisted
$V$-module $M.$ In this case, $M'$ is a $g^{-1}$-twisted $V$-module.
Moreover, $M$ is irreducible if and only if $M'$ is irreducible.

\end{remark}

Now we  review the notions of intertwining operators and fusion rules
from \cite{FHL}.

\begin{definition} Let $(V,\ Y)$ be a vertex operator algebra and
let $(W^{1},\ Y^{1}),\ (W^{2},\ Y^{2})$ and $(W^{3},\ Y^{3})$ be
$V$-modules. An \emph{intertwining operator} of type $\left(\begin{array}{c}
W^{3}\\
W^{1\ }W^{2}
\end{array}\right)$ is a linear map
\[
I(\cdot,\ z):\ W^{1}\to\text{\ensuremath{\mbox{Hom}(W^{2},\ W^{3})\{z\}}}
\]

\[
u\to I(u,\ z)=\sum_{n\in\mathbb{Q}}u_{n}z^{-n-1}
\]
 satisfying:

(1) for any $u\in W^{1}$ and $v\in W^{2}$, $u_{n}v=0$ for $n$
sufficiently large;

(2) $I(L(-1)v,\ z)=(\frac{d}{dz})I(v,\ z)$;

(3) (Jacobi identity) for any $u\in V,\ v\in W^{1}$

\[
z_{0}^{-1}\delta\left(\frac{z_{1}-z_{2}}{z_{0}}\right)Y^{1}(u,\ z_{1})I(v,\ z_{2})-z_{0}^{-1}\delta\left(\frac{-z_{2}+z_{1}}{z_{0}}\right)I(v,\ z_{2})Y^{3}(u,\ z_{1})
\]
\[
=z_{2}^{-1}\left(\frac{z_{1}-z_{0}}{z_{2}}\right)I(Y^{2}(u,\ z_{0})v,\ z_{2}).
\]

The space of all intertwining operators of type $\left(\begin{array}{c}
W^{3}\\
W^{1}\ W^{2}
\end{array}\right)$ is denoted by
$$I_{V}\left(\begin{array}{c}
W^{3}\\
W^{1}\ W^{2}
\end{array}\right).$$ Let $N_{W^{1},\ W^{2}}^{W^{3}}=\dim I_{V}\left(\begin{array}{c}
W^{3}\\
W^{1}\ W^{2}
\end{array}\right)$. These integers $N_{W^{1},\ W^{2}}^{W^{3}}$ are usually called the
\emph{fusion rules}. As usual, we use $M^{i'}$ to denote $\left(M^{i}\right)^{'}$,
the contragredient module of $M^{i}$.\end{definition}

%\begin{remark}\label{Intertwining expression} \cite{FZ} Let $M^{i}=\oplus_{n\in\mathbb{Z}}M^{i}(n),$
%$i=1,2,3$ be irreducible modules for a vertex operator algebra $V,$
%and the corresponding conformal weights are $h_{i}$, $i=1,2,3$.
%If $I(\cdot,z)$ is an intertwining operator of type $\left(\begin{array}{c}
%M^{3}\\
%M^{1}\ M^{2}
%\end{array}\right),$ then $I(\cdot,z)$ can be written as
%\[
%I(v,z)=\sum_{n\in\mathbb{Z}}v(n)z^{-n-1}z^{-h_{1}-h_{2}+h_{3}}
%\]
 %such that for honogeneous $v\in M^{1},$ $v(n)M^{2}(m)\subset M^{3}\left(m+deg\, v-1-n\right),$
%where $deg\, v=k$ means $v\in M^{1}(k).$

%We will write $o(v)=v(deg\, v-1).$

%\end{remark}

\begin{definition} Let $V$ be a vertex operator algebra, and $W^{1},$
$W^{2}$ be two $V$-modules. A module $(W,I)$, where $I\in I_{V}\left(\begin{array}{c}
\ \ W\ \\
W^{1}\ \ W^{2}
\end{array}\right),$ is called a \emph{tensor product} (or fusion product) of $W^{1}$
and $W^{2}$ if for any $V$-module $M$ and $\mathcal{Y}\in I_{V}\left(\begin{array}{c}
\ \ M\ \\
W^{1}\ \ W^{2}
\end{array}\right),$ there is a unique $V$-module homomorphism $f:W\rightarrow M,$ such
that $\mathcal{Y}=f\circ I.$ As usual, we denote $(W,I)$ by $W^{1}\boxtimes_{V}W^{2}.$
\end{definition}

The basic result is that the fusion product exists if $V$ is rational.
It is well known that if $V$ is rational, for any two irreducible
$V$-modules $W^{1},\ W^{2},$
\[
W^{1}\boxtimes_{V}W^{2}=\sum_{W}N_{W^{1},\ W^{2}}^{W}W
\]
 where $W$ runs over the set of equivalence classes of irreducible
$V$-modules.

It is well known that fusion rules have the following symmetry property
\cite{FHL}.

\begin{proposition}\label{fusion rule symmmetry property} Let $W^{i}$ $\left(i=1,2,3\right)$
be $V$-modules. Then

\[
N_{W^{1},W^{2}}^{W^{3}}=N_{W^{2},W^{1}}^{W^{3}},N_{W^{1},W^{2}}^{W^{3}}=N_{W^{1},(W^{3})'}^{(W^{2})'}.
\]
\end{proposition}

Now we recall some notions about quantum dimensions.

\begin{definition}Let $M=\oplus_{n\in\frac{1}{T}\mathbb{Z}_{+}}M_{\lambda+n}$
be a $g$-twisted $V$-module, the \emph{formal character} of $M$
is defined as

\[
\mbox{ch}_{q}M=\mbox{tr}_{M}q^{L\left(0\right)-c/24}=q^{\lambda-c/24}\sum_{n\in\frac{1}{T}\mathbb{Z}_{+}}\left(\dim M_{\lambda+n}\right)q^{n},
\]
where $c$ is the central charge of the vertex operator algebra $V$
and $\lambda$ is the conformal weight of $M$. \end{definition}

It is proved \cite{Z,DLM4} that $\mbox{ch}_{q}M$ converges to a
holomorphic function in the domain $|q|<1.$ We denote the holomorphic
function $\mbox{ch}_{q}M$ by $Z_{M}\left(\tau\right)$. Here and
below, $\tau$ is in the upper half plane $\mathbb{H}$ and $q=e^{2\pi i\tau}$.

Let $M^{0},\cdots,M^{d}$ be the inequivalent irreducible $V$-modules
with corresponding conformal weights $\lambda_{i}$ and $M^{0}\cong V$.
Define
\[
Z_{i}\left(u,v,\tau\right)=\mbox{tr}_{M^{i}}e^{2\pi i\left(v\left(0\right)\right)+\left(v,u\right)\text{/2}}q^{L\left(0\right)+u\left(0\right)+\left(u,u\right)/2-c/24}
\]
 for $u,v\in V_{1}$. Notice that if $u,v=0,$ $Z_{i}(u,v,\tau)=Z_{i}(\tau).$
Then we have the following theorem \cite{M,Z,DLM,DLM4}:

\begin{theorem} Let $V$ be a rational, $C_{2}$-cofinite vertex
operator algebra of CFT type. Assume $u,v\in V_{1}$ such that $u,v$
span an abelian Lie subalgebra of $V_{1}$. Let $\gamma=\left(\begin{array}{cc}
a & b\\
c & d
\end{array}\right)\in SL\left(2,\mathbb{Z}\right).$ Then $Z_{i}\left(u,v,\tau\right)$ converges to a holomorphic function
in the upper half plane and
\[
Z_{i}\left(u,v,\gamma\tau\right)=\sum_{j=0}^{d}\gamma_{i,j}Z_{j}\left(au+bv,cu+dv,\tau\right),
\]
 where $\gamma\tau=\frac{a\tau+b}{c\tau+d}$ and $\gamma_{i,j}\in\mathbb{C}$
are independent of the choice of $u,v$.  \end{theorem}

\begin{definition} In the case $\gamma=S=\left(\begin{array}{cc}
0 & -1\\
1 & 0
\end{array}\right)$, we have
\[
Z_{i}\left(u,v,-\frac{1}{\tau}\right)=\sum_{j=0}^{d}S_{i,j}Z_{j}\left(-v,u,\tau\right).
\]
The matrix $S=\left(S_{i,j}\right)_{i,j=0}^{d}$ is called \emph{$S$-matrix}.
\end{definition}

\begin{remark}\label{quan dim Si0/S00}Assume $V$ is a simple, rational
and $C_{2}$-cofinite vertex operator algebra of $CFT$ type with
$V\cong V'$. Let $M^{i}$ be as before where $M^{0}\cong V$. Also
assume $\min_{i}\left\{ \lambda_{i}\right\} =\lambda_{0}=0$ and $\lambda_{i}>0\ \forall i\not=0$.
Then $\mbox{qdim}_{V}M^{i}=\frac{S_{i,0}}{S_{0,0}}$ \cite{DJX}.
\end{remark}

The following theorem will play an important role in the last section
\cite{V,H}.

\begin{theorem} \label{Verlinde formula} Let $V$ be a rational
and $C_{2}$-cofinite simple vertex operator algebra of CFT type and
assume $V\cong V'$. Let $S=\left(S_{i,j}\right)_{i,j=0}^{d}$ be
the $S$-matrix as defined above. Then

(1) $\left(S^{-1}\right)_{i,j}=S_{i,j'}=S_{i',j}$, and $S_{i',j'}=S_{i,j};$

(2) $S$ is symmetric and $S^{2}=\left(\delta_{i,j'}\right)$;

(3) $N_{i,j}^{k}=\sum_{s=0}^{d}\frac{S_{i,s}S_{j,s}S_{s,k}^{-1}}{S_{0,s}}$.

\end{theorem}

% something about quantum diemensons @@@@@@@@@@@@@@@@@@@@@

We need the concept of quantum dimensions from \cite{DJX}.

\begin{definition} \label{quantum dimension}Let $V$ be a vertex
operator algebra and $M$ a $g$-twisted $V$-module such that $Z_{V}\left(\tau\right)$
and $Z_{M}\left(\tau\right)$ exists. The quantum dimension of $M$
over $V$ is defined as
\[
\mbox{qdim}_{V}M=\lim_{y\to0}\frac{Z_{M}\left(iy\right)}{Z_{V}\left(iy\right)},
\]
 where $y$ is real and positive. \end{definition}

From now on, we assume $V$ is a rational, $C_{2}$-cofinite vertex
operator algebra of CFT type with $V\cong V'$. Let $M^{0}\cong V,\, M^{1},\,\cdots,\, M^{d}$
denote all inequivalent irreducible $V$-modules. Moreover, we assume
the conformal weights $\lambda_{i}$ of $M^{i}$ are positive for
all $i>0.$ It is proved in \cite{DJ4} that $V_{L_{2}}^{A_{4}}$
satisfies all the assumptions.

Recall that simple module $M^{i}$ is called a \emph{simple
current }if $M^{i}\boxtimes M^{j}$ is simple $\forall j=0,\cdots,d.$ Here are some results on quantum dimensions \cite{DJX}.

\begin{proposition}\label{quantum-product} Let $V$ be a vertex operator algebra as before. Then

(1) $q\dim_{V}M^{i}\geq1,$ $\forall i=0,\cdots,d.$

(2)  For any $i,\, j=0,\cdots,\, d,$
\[
q\dim_{V}\left(M^{i}\boxtimes M^{j}\right)=q\dim_{V}M^{i}\cdot q\dim_{V}M^{j}
\]

(3) A $V$-module
$M$ is a simple current if and only if $q\dim_{V}M=1$.
\end{proposition}

\begin{theorem}\label{fix point order} Let $V$ be a rational and
$C_{2}$-cofinite simple vertex operator algebra, $G$ a finite subgroup
of $Aut(V).$ Also assume that $V$ is $g$-rational and the conformal
weight of any irreducible $g$-twisted $V$-module is positive except
for $V$ itself for all $g\in G$. Then $\mbox{qdim}_{V^{G}}V$ exists
and equals to $\left|G\right|$.

\end{theorem}

\section{The vertex operator algebra $V_{L_2}^{A_4}$}
\def\theequation{3.\arabic{equation}}
\setcounter{equation}{0}

Now we first briefly review the construction of rank one lattice vertex
operator algebra from \cite{FLM}. Then we recall some related results
about $V_{L}^{+}$ and $V_{L_{2}}^{A_{4}}$ from \cite{A1,A2,AD,ADL,DN1,DN2,DN3,DJL,FLM,DJ4}.
In the last part of this section, we also give the realization of
all irreducible $V_{L_{2}}^{A_{4}}$-modules.

\subsection{Construction of the vertex operator algebra $V_{L_{2}}^{A_{4}}$}

Let $L=\mathbb{Z}\alpha$ be a positive definite even lattice of rank
one, i.e., $\left(\alpha,\alpha\right)=2k$ for some positive integer
$k$. Set $\mathfrak{h}=L\otimes_{\mathbb{Z}}\mathbb{C}$ and extend
$\left(\cdot,\cdot\right)$ to a $\mathbb{C}$-bilinear form on $\mathfrak{h}$.
Let $\mathbb{C}\left[\mathfrak{h}\right]$ be the group algebra of
$\mathfrak{h}$ with a basis $\left\{ e^{\lambda}|\lambda\in\mathfrak{h}\right\} $.

Let $\hat{\mathfrak{h}}=\mathbb{C}\left[t,t^{-1}\right]\otimes\mathfrak{h}\oplus\mathbb{C}K$
be the corresponding Heisenberg algebra such that
\[
\left[\alpha\left(m\right),\alpha\left(n\right)\right]=2km\delta_{m+n,0}K\ \mbox{and\ }\left[K,\hat{\mathfrak{h}}\right]=0
\]
for any $m,n\in\mathbb{Z}$, where $\alpha\left(m\right)=\alpha\otimes t^{m}$.
Then $\hat{\mathfrak{h}}_{\ge0}=\mathbb{C}\left[t\right]\otimes\mathfrak{h}\oplus\mathbb{C}K$
is a subalgebra of $\mathfrak{\hat{h}}$ and the group algebra $\mathbb{C}\left[\mathfrak{h}\right]$
becomes a $\hat{\mathfrak{h}}_{\ge0}$-module by the action $\alpha\left(m\right)\cdot e^{\lambda}=\left(\lambda,\alpha\right)\delta_{m,0}e^{\lambda}$
and $K\cdot e^{\lambda}=e^{\lambda}$ for any $\lambda\in\mathfrak{h}$
and $m\ge0$. We denote by
\[
M\left(1,\lambda\right)=U\left(\hat{\mathfrak{h}}\right)\otimes_{U\left(\hat{\mathfrak{h}}_{\ge0}\right)}\mathbb{C}e^{\lambda}
\]
the $\hat{\mathfrak{h}}$-module induced from $\hat{\mathfrak{h}}_{\ge0}$-module
$\mathbb{C}e^{\lambda}$. Set $M\left(1\right)=M\left(1,0\right)$.
Then there exists a linear map $Y:M\left(1\right)\to\mbox{End}M\left(1\right)\left[\left[z,z^{-1}\right]\right]$
such that $\left(M\left(1\right),Y,\mathbf{1},\omega\right)$ carries
a simple vertex operator algebra structure and $M\left(1,\lambda\right)$
becomes an irreducible $M\left(1\right)$-module for $\lambda\in\mathfrak{h}$
\cite{FLM}. Let $\mathbb{C}\left[L\right]$ be the group algebra
of $L$ with a basis $e^{\alpha}$ for $\alpha\in L$. The lattice
vertex operator algebra associated to $L$ is given by
\[
V_{L}=M\left(1\right)\otimes\mathbb{C}\left[L\right].
\]

The dual lattice $L^{\circ}$ of $L$ is
\[
L^{\circ}=\left\{ \lambda\in\mathfrak{h}|\left(\alpha,\lambda\right)\in\mathbb{Z}\right\} =\frac{1}{2k}L.
\]

Then $L^{\circ}=\cup_{i=-k+1}^{k}\left(L+\lambda_{i}\right)$ is the
coset decomposition with $\lambda_{i}=\frac{i}{2k}\alpha$. Set $V_{L+\lambda_{i}}=M\left(1\right)\otimes\mathbb{C}\left[L+\lambda_{i}\right]$.
Then $V_{L+\lambda_{i}}$ for $i=-k+1$, $\cdots$, $k$ are all the
inequivalent irreducible modules for $V_{L}$ \cite{B,FLM,D1}.

Let $\theta$ be a linear isomorphism of $V_{\mathfrak{h}}$ defined
by
\[
\theta\left(\alpha\left(-n_{1}\right)\cdots\alpha\left(-n_{k}\right)\otimes e^{\lambda}\right)=\left(-1\right)^{k}\alpha\left(-n_{1}\right)\cdots\alpha\left(-n_{k}\right)\otimes e^{-\lambda},
\]
 for $n\in\mathbb{Z}_{+}$ and $\lambda\in\mathfrak{h}$. Then $\theta$
induces automorphisms of $V_{L}$ and $M\left(1\right)$. For a $\theta$-invariant
subspace $W$ of $V_{\mathfrak{h}}=M\left(1\right)\otimes\mathbb{C}\left[\mathfrak{h}\right]$,
we denote the $\pm1$-eigenspaces of $W$ for $\theta$ by $W^{\pm}$.
Then $\left(V_{L}^{+},Y,\mathbf{1},\omega\right)$ and $\left(M\left(1\right)^{+},Y,\mathbf{1},\omega\right)$
are vertex operator algebras.

Now we recall the construction of $\theta$-twisted $V_{L}$-modules
\cite{FLM,D2}. Let $\hat{\mathfrak{h}}\left[-1\right]=\mathfrak{h}\otimes t^{1/2}\mathbb{C}\left[t,t^{-1}\right]\oplus\mathbb{C}K$
be a Lie algebra with the commutation relation
\[
\left[\alpha\otimes t^{m},\alpha\otimes t^{n}\right]=m\delta_{m+n,0}\left(\alpha,\alpha\right)K\ \mbox{and}\ \left[K,\hat{\mathfrak{h}}\left[-1\right]\right]=0\
\]
for $m,n\in1/2+\mathbb{Z}$. Then there is a one-dimensional module
for $\hat{\mathfrak{h}}\left[-1\right]_{+}=\mathfrak{h}\otimes t^{1/2}\mathbb{C}\left[t\right]\oplus\mathbb{C}K$,
which could be identified with $\mathbb{C}$, by the action
\[
\left(\alpha\otimes t^{m}\right)\cdot1=0\ \mbox{and}\ K\cdot1=1\ \mbox{for}\ m\in1/2+\mathbb{N}.
\]
 Set $M\left(1\right)\left(\theta\right)$ the induced $\hat{\mathfrak{h}}\left[-1\right]$-module:
\[
M\left(1\right)\left(\theta\right)=U\left(\hat{\mathfrak{h}}\right)\left[-1\right]\otimes_{U\left(\hat{\mathfrak{h}}\left[-1\right]_{+}\right)}\mathbb{C}.
\]

Let $\chi_{s}$ be a character of $L/2L$ such that $\chi_{s}\left(\alpha\right)=\left(-1\right)^{s}$
for $s=0,1$ and $T_{\chi_{s}}=\mathbb{C}$ the irreducible $L/2L$-module
with character $\chi_{s}$. Then $V_{L}^{T_{s}}=M\left(1\right)\left(\theta\right)\otimes T_{\chi_{s}}$
is an irreducible $\theta$-twisted $V_{L}$-module. We denote the
$\pm1$-eigenspaces of $V_{L}^{T_{s}}$ under $\theta$ by $\left(V_{L}^{T_{s}}\right)^{\pm}$.
Then we have the following result:

\begin{theorem} Any irreducible $V_{L}^{+}$-module is isomorphic
to one of the following modules:

\[
V_{L}^{\pm},V_{\lambda_{i}+L}\left(i\not=k\right),V_{\lambda_{k}+L}^{\pm},\left(V_{L}^{T_{s}}\right)^{\pm}.
\]

\end{theorem}

Let $L_{2}=\mathbb{Z}\alpha$ be the rank one positive-definite even
lattice such that $\left(\alpha,\alpha\right)=2$ and $V_{L_{2}}$
the associated simple rational vertex operator algebra. Then $\left(V_{L_{2}}\right)_{1}\cong sl_{2}\left(\mathbb{C}\right)$
and $\left(V_{L_{2}}\right)_{1}$ has an orthonormal basis:

\[
x^{1}=\frac{1}{\sqrt{2}}\alpha\left(-1\right)\mathbf{1},\ x^{2}=\frac{1}{\sqrt{2}}\left(e^{\alpha}+e^{-\alpha}\right),\ x^{3}=\frac{i}{\sqrt{2}}\left(e^{\alpha}-e^{-\alpha}\right).
\]
For $x\in \left(V_{L_{2}}\right)_{1}$ we also use $x(n)$ for $x_n$ for $n\in\Z.$
Let $\sigma,\tau_{i}\in Aut\left(V_{L_{2}}\right)$, $i=1,2,3$ be
such that

\[
\sigma\left(x^{1},x^{2},x^{3}\right)=\left(x^{1},x^{2},x^{3}\right)\left[\begin{array}{ccc}
0 & 1 & 0\\
0 & 0 & -1\\
-1 & 0 & 0
\end{array}\right].
\]

\[
\tau_{1}\left(x^{1},x^{2},x^{3}\right)=\left(x^{1},x^{2},x^{3}\right)\left[\begin{array}{ccc}
1\\
 & -1\\
 &  & -1
\end{array}\right],
\]

\[
\tau_{2}\left(x^{1},x^{2},x^{3}\right)=\left(x^{1},x^{2},x^{3}\right)\left[\begin{array}{ccc}
-1\\
 & 1\\
 &  & -1
\end{array}\right],
\]

\[
\tau_{3}\left(x^{1},x^{2},x^{3}\right)=\left(x^{1},x^{2},x^{3}\right)\left[\begin{array}{ccc}
-1\\
 & -1\\
 &  & 1
\end{array}\right].
\]

Then $\sigma$ and $\tau_{i},i=1,2,3,$ generate a finite subgroup
of $Aut\left(V_{L_{2}}\right)$ isomorphic to the alternating group
$A_{4}$. We simply denote this group by $A_{4}$. It is easy to check
that the subgroup $K$ generated by $\tau_{i},i=1,2,3$, is a normal
subgroup of $A_{4}$ of order $4$. Let $\beta=2\alpha$.

The following result can be found in \cite{DG}.
\begin{lemma} We have $V_{L_{2}}^{K}= V_{\mathbb{Z}\beta}^{+}$
 and $V_{L_{2}}^{A_{4}}=\left(V_{\mathbb{Z}\beta}^{+}\right)^{\left\langle \sigma\right\rangle }$.
\end{lemma}

By \cite{DM1}, there is decomposition
\begin{equation}\label{Decomp1}
V_{\mathbb{Z}\beta}^{+}=\left(V_{\mathbb{Z}\beta}^{+}\right)^{0}\oplus
\left(V_{\mathbb{Z}\beta}^{+}\right)^{1}\oplus\left(V_{\mathbb{Z}\beta}^{+}\right)^{2}
\end{equation}
where $\left(V_{\mathbb{Z}\beta}^{+}\right)^{0}=\left(V_{\mathbb{Z}\beta}^{+}\right)^{\left\langle \sigma\right\rangle }$
is a simple vertex operator algebra and $\left(V_{\mathbb{Z}\beta}^{+}\right)^{i}$
is an irreducible $\left(V_{\mathbb{Z}\beta}^{+}\right)^{0}$-module,
$i=1,2$. Similarly, as a $\left(V_{\mathbb{Z}\beta}^{+}\right)^{\left\langle \sigma\right\rangle }$-module,
we have
\begin{equation}\label{Decomp2}
V_{\mathbb{Z}+\frac{1}{4}\beta}=V_{\mathbb{Z}+\frac{1}{4}\beta}^{0}
\oplus V_{\mathbb{Z}+\frac{1}{4}\beta}^{1}\oplus V_{\mathbb{Z}+\frac{1}{4}\beta}^{2}
\end{equation}
 such that $V_{\mathbb{Z}+\frac{1}{4}\beta}^{i}$ is
irreducible $\left(V_{\mathbb{Z}\beta}^{+}\right)^{\left\langle \sigma\right\rangle }$-module,
$i=0,1,2$  \cite{DY}.  The details of the realization of $V_{\mathbb{Z}+\frac{1}{4}\beta}^{i}$
will be provided in the next subsection.

Let $W_{\sigma^{i},1}$, $W_{\sigma^{i},2}$ be the  two irreducible
$\sigma^{i}$-twisted modules of $V_{\mathbb{Z}\beta}^{+}$, $i=1,2$
\cite{DJ4}. Then each $W_{\sigma^{i},j}$ is a direct sum of irreducible $\left(V_{\mathbb{Z}\beta}^{+}\right)^{\left\langle \sigma\right\rangle }$-submodules
$W_{\sigma^{i},j}^{k}$ for $k=0,1,2$ . There are
exactly 21 irreducible modules of $\left(V_{\mathbb{Z}\beta}^{+}\right)^{\left\langle \sigma\right\rangle }$
which is listed as following \cite{DJ4}:
\begin{equation}\label{irreducible}
\left\{ \left(V_{\mathbb{Z}\beta}^{+}\right)^{m},V_{\mathbb{Z}\beta}^{-},V_{\mathbb{Z}\beta+\frac{1}{8}\beta},V_{\mathbb{Z}\beta+\frac{3}{8}\beta},W_{\sigma^{i},j}^{k},V_{\mathbb{Z}\beta+\frac{1}{4}\beta}^{n}|m,n,k=0,1,2;i,j=1,2.\right\} \end{equation}
Here $\mbox{\ensuremath{\left(V_{\mathbb{Z}\beta}^{+}\right)^{m}}}$
is the eigenspace of $\sigma$ with eigenvalue $e^{\frac{2\pi im}{3}}$.

\subsection{Realizations of the Irreducible $V_{L_{2}}^{A_{4}}$-modules}

Let $\sigma,$ $\tau_{i}$ and $x^{i},i=1,2,3$ be as before. Set
\[
h=\frac{1}{3\sqrt{6}}\left(x^{1}+x^{2}-x^{3}\right),
\]
\[
y^{1}=\frac{1}{\sqrt{3}}\left(x^{1}+\frac{-1+\sqrt{3}i}{2}x^{2}+\frac{1+\sqrt{3}i}{2}x^{3}\right),
\]
\[
y^{2}=\frac{1}{\sqrt{3}}\left(x^{1}+\frac{-1-\sqrt{3}i}{2}x^{2}+\frac{1-\sqrt{3}i}{2}x^{3}\right).
\]
Then
\[
L\left(n\right)h=\delta_{n,0}h,\ \ h\left(n\right)h=\frac{1}{18}\delta_{n,1}\mathbf{1},\ n\in\mathbb{Z},
\]
\[
h\left(0\right)y^{1}=\frac{1}{3}y^{1},\ h\left(0\right)y^{2}=-\frac{1}{3}y^{2},\ y^{1}\left(0\right)y^{2}=6h.
\]
(see \cite{DJ4}).
It follows that $h\left(0\right)$ acts semisimply on $V_{L_{2}}$
with rational eigenvalues. So $e^{2\pi ih\left(0\right)}$ is an automorphism
of $V_{L_{2}}$ of finite order \cite{DG,L2}. Since

\[
e^{2\pi ih\left(0\right)}h=h,\ \ e^{2\pi ih\left(0\right)}y^{1}=\frac{-1+\sqrt{3}i}{2}y^{1},\ \ e^{2\pi ih\left(0\right)}y^{2}=\frac{-1-\sqrt{3}i}{2}y^{2},
\]
it is easy to see that
\begin{equation}
e^{2\pi ih\left(0\right)}=\sigma.\label{sigma and h}
\end{equation}

The action of the group generated by $\sigma,$ $\tau_{i}$, $i=1,2,3$
on $V_{L_{2}}$ is isomorphic to alternating group $A_{4}.$ Actually,
$\sigma=e^{2\pi ih(0)}$ and $\tau_{i}=e^{\pi ix_{i}(0)}$ ($i=1,2,3$)
also act on $V_{\mathbb{Z}\frac{1}{2}\alpha}=M(1)\otimes\mathbb{C}[\frac{1}{2}\mathbb{Z}\alpha]$,
where the action of the group $\langle\sigma,\tau_{i}|i=1,2,3\rangle$
on $V_{\Z\frac{1}{2}\alpha}$ is isomorphic to $SL(2,3),$ the special
linear group of degree 2 over a field of three elements. Thus by the
quantum Galois theory \cite{DM1},
\begin{equation}
V_{\mathbb{Z}\frac{\alpha}{2}}\cong\bigoplus_{\chi}V_{\chi}\otimes W_{\chi}\label{eq:quantum galois}
\end{equation}
where $\chi$ runs over all irreducible characters of $SL(2,3).$ The
irreducible representations of the group $SL(2,3)$ are well known:
three 1-dimensional, one 3-dimensional and three 2-dimensional irreducible
representations. We denote them by $U_{1}^{k},$ $U_{3}$ and $U_{2}^{k},$
$k=0,1,2$ respectively, where the subindex $i$ is the dimension
of the module and the upper indices distinguish the irreducible modules
of the same dimension. The irreducible modules with the same dimension
can be distinguished by the eigenvalues of the action of $\sigma$:
$\sigma|_{U_{1}^{k}}=e^{\frac{2\pi ik}{3}}$, $\sigma$ has eigenvalues
$ $$e^{\frac{2\pi i}{6}+\frac{2\pi ik}{3}}$ and $ $$e^{-\frac{2\pi i}{6}+\frac{2\pi ik}{3}}$
on $U_{2}^{k}$ and the eigenvalues of $\sigma$ on $U_{3}$ are the
three cube roots of unity.

The set of scalar matrices of $SL(2,3)$ is a normal subgroup isomorphic to $\mathbb{Z}_{2}$
and $A_{4}\cong SL(2,3)/\mathbb{Z}_{2}.$
The group $A_{4}$ has three $1$-dimensional and one $3$-dimensional
irreducible modules. Thus $ $$V_{L_{2}}$ and $V_{L_{2}+\frac{1}{2}\alpha}$
can be decomposed as $V_{L_{2}}^{A_{4}}$-modules \cite{DY}:

\[
V_{L_{2}}=(V_{\mathbb{Z}\beta}^{+})^{0}\otimes U_{1}^{0}\oplus(V_{\mathbb{Z}\beta}^{+})^{1}\otimes U_{1}^{1}\oplus(V_{\mathbb{Z}\beta}^{+})^{2}\otimes U_{1}^{2}\oplus V_{\mathbb{Z}\beta}^{-}\otimes U_{3},
\]

\[
V_{L_{2}+\frac{1}{2}\alpha}=V_{\mathbb{Z\beta}+\frac{1}{4}\beta}^{0}\otimes U_{2}^{0}\oplus V_{\mathbb{Z\beta}+\frac{1}{4}\beta}^{1}\otimes U_{2}^{1}\oplus V_{\mathbb{Z\beta}+\frac{1}{4}\beta}^{2}\otimes U_{2}^{2}.
\]
Some of those $V_{L_{2}}^{A_{4}}$-modules listed in (\ref{irreducible})
can be realized differently by considering the orbifold vertex operator
algebra $V_{L_{2}}^{\left\langle \sigma\right\rangle }$.

\begin{proposition} \label{inner automorphism}
Let $g$ be an automorphism of $V_{L_{2}}$ of order $T\text{\ensuremath{\neq}}1.$
Then there exists some vector $u\in(V_{L_{2}})_{1},$ such that $g=e^{2\pi iu(0)}.$\end{proposition}

\proof The vertex operator algebra $V_{L_{2}}$ is isomorphic to
the affine vertex operator algebra associated to the simple Lie algebra
${sl}_{2}(\C)$ of level 1. We know that $Aut(V_{L_{2}})\cong Aut({sl}_{2}(\C)).$
The restriction of $g$ on $\left(V_{L_{2}}\right)_{1}$ is also an
isomorphism of order $T,$ which has an eigenspaces decomposition.

\emph{Claim: There exists a unique (up to a scalar) nonzero vector
$a\in\left(V_{L_{2}}\right)_{1}$ such that $ga=a.$ }

Note that $(a,b)=-a(1)b$ defines a nondegenerate symmetric invariant bilinear form on $\left(V_{L_{2}}\right)_{1},$ since $\left(V_{L_{2}}\right)_{1}$ is isomorphic to
${sl}_{2}(\C)$ which has a unique nondegenerate symmetric invariant bilinear form
up to a constant. This implies that $(ga,gb)=(a,b)$ for all $a,b\in \left(V_{L_{2}}\right)_{1}$
and $\left(V_{L_{2}}\right)_{1}^\rho$ and   $\left(V_{L_{2}}\right)_{1}^{\bar \rho}$
have the same dimensions where $\left(V_{L_{2}}\right)_{1}^\rho$ is the eigenspace of $g$ on
 $\left(V_{L_{2}}\right)_{1}$ with eigenvalue $\rho$ which is a root of unity.

First,  $g\big|_{\left(V_{L_{2}}\right)_{1}}$ is not a constant. Since $\left(V_{L_{2}}\right)_{1}$ is a simple Lie algebra we see that $\left(V_{L_{2}}\right)_{1}$ is spanned by $a(0)b$ for $a,b\in \left(V_{L_{2}}\right)_{1}.$ If $g$ acts on $\left(V_{L_{2}}\right)_{1}$ as a constant $\rho.$ Then $g$ also acts as $\rho^2.$ This forces $\rho=1,$ a contradiction.

If there are exactly  two eigenvalues $\rho_1,\rho_2$  of $g$ on $(V_{L_{2}})_{1},$ we
deduce that $\rho_{1}=\pm1,\,\rho_{2}=\mp1.$ Otherwise  $\rho_{1}=\overline{\rho_{2}}\neq\pm1$ and
 $\left(V_{L_{2}}\right)_{1}$ has even dimension, a contradiction.  Without loss of generality,
assume $\rho_{1}=1,$ $\rho_{2}=-1.$ If the eigenspace of $g$ with
eigenvalue 1 is two dimensional, then the eigenspace of $g$ on $[\left(V_{L_{2}}\right)_{1},\left(V_{L_{2}}\right)_{1}]=\left(V_{L_{2}}\right)_{1}$ with eigenvalue $1$ is one dimensional, a contradiction.

The only case left is that $\rho_{1},\,\rho_{2},\,\rho_{3}$ are three
distinct eigenvalues of $g$ on $(V_{L_{2}})_{1}.$ Assume that $\rho_{1}=\overline{\rho_{2}}.$
Using the fact that $[\left(V_{L_{2}}\right)_{1},\left(V_{L_{2}}\right)_{1}]=\left(V_{L_{2}}\right)_{1}$ we see that $\rho_3=1$ and each eigenspace is one dimensional. The claim is proved.

Let $a\in (V_{L_{2}})_{1}$ be an eigenvector of $g$ with eigenvalue $1.$ Consider the Jordan decomposition of  $a=a_{s}+a_{n},$ where $a_{s}$ and $a_{n}$ are the
semisimple part and nilpotent part of $a.$ It is easy to see that
$a$ is not nilpotent due to the eigenspace decomposition, and $a_{s}$
is also a fixed point of $g$ since $a$ is a fixed point of $g$.
Since the fixed point space is 1 dimensional, $a=a_{s}$, which acts
semisimply on $(V_{L_{2}})_{1}.$ The structure of ${sl}_{2}(\C)$
tells us that there exists $\gamma a$, for some $\gamma\in\mathbb{C}^{*},$
such that $(\gamma a,\gamma a)=2,$ $[\gamma a,e^{\gamma a}]=2e^{\gamma a},$
$ $$[\gamma a,e^{-\gamma a}]=-2e^{-\gamma a},$ $g(e^{\gamma a})=e^{2\pi i\frac{r}{T}}e^{\gamma a}$
and $g(e^{-\gamma a})=e^{-2\pi i\frac{r}{T}}e^{-\gamma a}.$ It is
clear that $g=e^{\pi i\frac{r}{T}\gamma a(0)},$ i.e. $u=\frac{r}{2T}\gamma a.$
$ $\endproof

\begin{remark} The group $SO(3)$ is the connected compact subgroup
of $Aut(V_{L_{2}}),$ whose discrete subgroup are the cyclic group
$Z_{n},$ the dihedral group $D_{n},$ $A_{4},$ $S_{4}$ and $A_{5}.$
The above proposition indicates that the orbifold vertex operator
algebra $V_{L_{2}}^{Z_{n}}\cong V_{\mathbb{Z}n\alpha}.$ One could
also get $V_{L_{2}}^{D_{n}}\cong V_{\mathbb{Z}n\alpha}^{+}.$
\end{remark}

\begin{remark}\label{A4 12dim} It is worthy to point out that for
any $g\in Aut(V_{L_{2}})$ of finite order $T$, the $g$-twisted
module category is equivalent to the category of ordinary modules.
Thus $V_{L_{2}}$ is $g$-rational for any such $g.$ Following from
Theorem \ref{fix point order}, q$\dim_{V_{L_{2}}^{A_4}} V_{L_{2}}=o(A_{4})=12.$
\end{remark}

In our case, we have $V_{L_{2}}^{\left\langle \sigma\right\rangle }\cong V_{\mathbb{Z\gamma}}\cong V_{\mathbb{Z}3\alpha},$
$(\gamma,\gamma)=18,$ i.e $\gamma=18h.$ One immediately gets that
$V_{L_{2}}\cong V_{\mathbb{Z}\gamma}\oplus V_{\mathbb{Z}\gamma+\frac{1}{3}\gamma}\oplus V_{\mathbb{Z}\gamma+\frac{2}{3}\gamma}$
and $V_{L_{2}+\frac{1}{2}\alpha}\cong V_{\mathbb{Z}\gamma+\frac{1}{6}\gamma}\oplus V_{\mathbb{Z}\gamma+\frac{1}{2}\gamma}\oplus V_{\mathbb{Z}-\frac{1}{6}\gamma}$
due to the eigenvalues of the $\sigma$ action on $V_{\mathbb{Z}\frac{1}{2}\alpha}$.
The eigenvalues of $\sigma$ on $V_{L_{2}}$ and $V_{L_{2}+\frac{1}{2}\alpha}$
(see equation (\ref{eq:quantum galois})) give us the following proposition.

\begin{proposition} \label{non-twisted decomposition}As $V_{L_{2}}^{A_{4}}$-modules,
we have the following identification:
\[
V_{\mathbb{Z}\gamma}\cong(V_{\mathbb{Z}\beta}^{+})^{0}+V_{\mathbb{Z}\beta}^{-},
\]
\[
V_{\mathbb{Z}\gamma+\frac{1}{3}\gamma}\cong(V_{\mathbb{Z}\beta}^{+})^{1}+V_{\mathbb{Z}\beta}^{-},
\]
\[
V_{\mathbb{Z}\gamma+\frac{2}{3}\gamma}\cong(V_{\mathbb{Z}\beta}^{+})^{2}+V_{\mathbb{Z}\beta}^{-},
\]
\[
V_{\mathbb{Z}\text{\ensuremath{\gamma}}+\frac{1}{6}\gamma}=V_{\mathbb{Z\beta}+\frac{1}{4}\beta}^{0}+V_{\mathbb{Z\beta}+\frac{1}{4}\beta}^{1},
\]
\[
V_{\mathbb{Z}\gamma+\frac{1}{2}\gamma}=V_{\mathbb{Z\beta}+\frac{1}{4}\beta}^{1}+V_{\mathbb{Z\beta}+\frac{1}{4}\beta}^{2},
\]
\[
V_{\mathbb{Z}\gamma-\frac{1}{6}\gamma}=V_{\mathbb{Z\beta}+\frac{1}{4}\beta}^{0}+V_{\mathbb{Z\beta}+\frac{1}{4}\beta}^{2}.
\]
\end{proposition}

Now we briefly review the irreducible $V_{L_{2}}^{A_{4}}$-modules
which are constructed from the $\sigma^{i}$-twisted $V_{L_{2}}$-modules.
Let $W^{1}\cong V_{L_{2}},\, W^{2}\cong V_{L_{2}+\frac{1}{2}\alpha}.$
Set
\[
w^{1}=e^{\frac{\alpha}{2}}+\frac{(\sqrt{3}-1)(1+i)}{2}e^{-\frac{\alpha}{2}},\, w^{2}=\frac{1}{\sqrt{2}}\left[(\sqrt{3}-1)e^{\frac{\alpha}{2}}-(1+i)e^{-\frac{\alpha}{2}}\right].
\]
For any $u\in(V_{L_{2}})_{1}$ such that $g=e^{2\pi iu(0)}$ is an
automorphism of $V_{L_{2}}$ of finite order, $ $ define
\[
\Delta\left(u,z\right)=z^{h\left(0\right)}\mbox{exp}\left(\sum_{k=1}^{\infty}\frac{u\left(k\right)}{-k}\left(-z\right)^{-k}\right).
\]
It is proved in \cite{L1} that $(W^{i,T_{1}},Y_{g}(\cdot,z))=(W^{i},Y(\Delta\left(u,z\right)\cdot,z))$
are irreducible $g$-twisted modules of $V_{L_{2}},$ $i=1,2.$ The
$\sigma^{i}$-twisted $V_{L_{2}}$-modules were constructed in \cite{DJ4}
following this idea, where the twisted vertex operator was also determined.

For the $\sigma$-twisted $V_{L_{2}}$-modules,

\[
\Delta(h,z)L(-2){\bf 1}=L(-2){\bf 1}+z^{-1}h(-1)\mathbf{1}+\frac{1}{36}z^{-2}\mathbf{1,}
\]

\[
Y_{\sigma}(h,z)=Y(h+\frac{1}{18}z^{-1},z),
\]

\[
Y_{\sigma}(y^{1},z)=z^{\frac{1}{3}}Y(y^{1},z),
\]

\[
Y_{\sigma}(y^{2},z)=z^{-\frac{1}{3}}Y(y^{2},z).
\]

For the $\sigma^{2}$-twisted $V_{L_{2}}$-modules,

\[
\Delta(-h,z)L(-2){\bf 1}=L(-2){\bf 1}-z^{-1}h(-1)\mathbf{1}+\frac{1}{36}z^{-2}\mathbf{1,}
\]

\[
Y_{\sigma^{2}}(h,z)=Y(-h+\frac{1}{18}z^{-1},z),
\]

\[
Y_{\sigma^{2}}(y^{1},z)=z^{-\frac{1}{3}}Y(y^{1},z),
\]

\[
Y_{\sigma^{2}}(y^{2},z)=z^{\frac{1}{3}}Y(y^{2},z).
\]

The irreducible $V_{L_{2}}^{A_{4}}$-modules $W_{\sigma^{i},j}^{k}$
are realized in the $\sigma^{i}$-twisted $V_{L_{2}}$-module \cite{DJ4}.
	The following table gives the conformal weight of each $W_{\sigma^{i},j}^{k}$ \cite{DJ4}:

\begin{tabular}{|c|c|c|c|c|c|c|}
\hline
 & $W_{\sigma,1}^{0}$$ $ & $W_{\sigma,1}^{1}$ & $W_{\sigma,1}^{2}$ & $W_{\sigma,2}^{0}$$ $ & $W_{\sigma,2}^{1}$ & $W_{\sigma,2}^{2}$\tabularnewline
\hline
$L(0)$ & $\text{\ensuremath{\frac{1}{36}}}$ & $\frac{}{}$$\frac{25}{36}$ & $\frac{}{}$$\frac{}{}$$\frac{}{}$$\frac{}{}$$\frac{49}{36}$ & $\frac{}{}$$\frac{}{}$$\frac{1}{9}$ & $\frac{}{}$$\frac{}{}$$\frac{4}{9}$ & $\frac{}{}$$\frac{}{}$$\frac{16}{9}$\tabularnewline
\hline
lowest weight vector & $\mathbf{1}$ & $y^{2}$ & $y^{1}$ & $w^{2}$ & $w^{1}$ & $y^{2}(-2)w^{2}$\tabularnewline
\hline
\hline
 & $W_{\sigma^{2},1}^{0}$$ $ & $W_{\sigma^{2},1}^{1}$ & $W_{\sigma^{2},1}^{2}$ & $W_{\sigma^{2},2}^{0}$$ $ & $W_{\sigma^{2},2}^{1}$ & $W_{\sigma^{2},2}^{2}$\tabularnewline
\hline
$L(0)$ & $\text{\ensuremath{\frac{1}{36}}}$ & $\frac{}{}$$\frac{25}{36}$ & $\frac{}{}$$\frac{}{}$$\frac{}{}$$\frac{}{}$$\frac{49}{36}$ & $\frac{}{}$$\frac{}{}$$\frac{1}{9}$ & $\frac{}{}$$\frac{}{}$$\frac{4}{9}$ & $\frac{}{}$$\frac{}{}$$\frac{16}{9}$\tabularnewline
\hline
lowest weight vector & $\mathbf{1}$ & $y^{1}$ & $y^{2}$ & $w^{1}$ & $w^{2}$ & $y^{1}(-2)w^{1}$\tabularnewline
\hline
\end{tabular}

\begin{proposition}\label{twisted modules identification}

As $V_{L_{2}}^{A_{4}}$-modules, we have the following identification:

\[
V_{\mathbb{Z\gamma}+\frac{1}{18}\gamma}\cong W_{\sigma,1}^{0},\ V_{\mathbb{Z}\gamma-\frac{5}{18}\gamma}\cong W_{\sigma,1}^{1},\ V_{\mathbb{Z\gamma}+\frac{7}{18}\gamma}\cong W_{\sigma,1}^{2},
\]
\[
V_{\mathbb{Z\gamma}-\frac{1}{9}\gamma}\cong W_{\sigma,2}^{0},\ V_{\mathbb{Z\gamma}+\frac{2}{9}\gamma}\cong W_{\sigma,2}^{1},\ V_{\mathbb{Z\gamma}-\frac{4}{9}\gamma}\cong W_{\sigma,2}^{2},
\]
\[
V_{\mathbb{Z\gamma}-\frac{1}{18}\gamma}\cong W_{\sigma^{2},1}^{0},\ V_{\mathbb{Z\gamma}+\frac{5}{18}\gamma}\cong W_{\sigma^{2},1}^{1},\ V_{\mathbb{Z}\gamma-\frac{7}{18}\gamma}\cong W_{\sigma^{2},1}^{2},
\]
\[
V_{\mathbb{Z\gamma}\text{+}\frac{1}{9}\gamma}\cong W_{\sigma^{2},2}^{0},\ V_{\mathbb{Z\gamma}-\frac{2}{9}\gamma}\cong W_{\sigma^{2},2}^{1},\ V_{\mathbb{Z\gamma}+\frac{4}{9}\gamma}\cong W_{\sigma^{2},2}^{2}.
\]

\end{proposition}

\proof We only prove the first isomorphism. We know that the $V_{L_{2}}^{A_{4}}$-module
in $V_{L_{2}}$ generated by $y^{1}$ is isomorphic to $\left(V_{\mathbb{Z}\beta}^{+}\right)^{1}\subset V_{\mathbb{Z}\gamma+\frac{1}{3}\gamma}.$
Since the conformal weight of $W{}_{\sigma,1}^{0}$ is $\frac{1}{36},$
$W{}_{\sigma,1}^{0}\cong V_{\mathbb{Z}\gamma+\frac{1}{18}\gamma}$
or $ $$V_{\mathbb{Z}\gamma-\frac{1}{18}\gamma}.$ The twisted vertex
operator $Y_{\sigma}(y^{1},z)$ would help us to determine which is
the right isomorphism. We have
\[
Y_{\sigma}(y^{1},z)\mathbf{1}=z^{\frac{1}{3}}Y(y^{1},z)\mathbf{1\subset}W_{\sigma,1}^{2},
\]
 where the conformal weight of $W_{\sigma,1}^{2}=\frac{49}{36}$.
On the other hand, the fusion rules among irreducible $V_{\mathbb{Z}\gamma}$-modules
are as following:
\[
V_{\mathbb{Z}\gamma+\frac{1}{3}\gamma}\boxtimes V_{\mathbb{Z}\gamma+\frac{1}{18}\gamma}=V_{\mathbb{Z}\gamma+\frac{7}{18}\gamma},\ V_{\mathbb{Z}\gamma+\frac{1}{3}\gamma}\boxtimes V_{\mathbb{Z}\gamma-\frac{1}{18}\gamma}=V_{\mathbb{Z}\gamma+\frac{5}{18}\gamma}.
\]
 We already mentioned that $y^{1}\in\left(V_{\mathbb{Z}\beta}^{+}\right)^{1}\subset V_{\mathbb{Z}\gamma+\frac{1}{3}\gamma}.$
Comparing the conformal weights of $W_{\sigma,1}^{2},$ $V_{\mathbb{Z}\gamma+\frac{7}{18}\gamma}$
and $V_{\mathbb{Z}\gamma+\frac{5}{18}\gamma}$ tells us that $W_{\sigma,1}^{0}\cong V_{\mathbb{Z}\gamma+\frac{1}{18}\gamma}.$

Other isomorphisms could be proved using a similar argument. \endproof

\section{Quantum dimensions of irreducible $V_{L_2}^{A_4}$-modules}
\def\theequation{4.\arabic{equation}}
\setcounter{equation}{0}

In this section, we determine the quantum dimensions of all irreducible $V_{L_{2}}^{A_{4}}$-modules.
We first investigate properties of quantum dimensions of irreducible
twisted $V_{L_{2}}^{A_{4}}$-modules.

Let $V$ be a vertex operator algebra and Let $g=e^{2\pi ih\left(0\right)}$
be an automorphism of $V$ of finite order where $h\in V_1$ such that $h(0)$ acts on $V$ semisimply.
Let $M$ be an irreducible $V$-module. By \cite{DLM1,L1} we see
that $\left(M^{g},Y_{g}\left(\cdot,z\right)\right)=\left(M,Y_{M}\left(\Delta\left(h,z\right)\cdot,z\right)\right)$
is a $g$-twisted $V$-module, where $\Delta\left(h,z\right)=z^{h\left(0\right)}\mbox{exp}\left(\sum_{k=1}^{\infty}\frac{h\left(k\right)}{-k}\left(-z\right)^{-k}\right)$.

\begin{proposition} \label{twisted quantum} Let $V$ be a rational,
$C_{2}$-cofinite vertex operator algebra with central charge $c$
and $M^{0},\cdots,M^{d}$ be the inequivalent irreducible $V$-modules
with $M^{0}\cong V$ and the corresponding conformal weights $\lambda_{i}>0$
for $0<i\le d$. Let $g$ be as defined.
Then $\mbox{qdim}M^{i}=\mbox{qdim}\left(M^{i}\right)^{g}$,
$0\le i\le d$.\end{proposition}

\begin{proof}

The $q$-character of $\left(M^{i}\right)^{g}$ are given by
\[
\mbox{ch}_{q}\left(M^{i}\right)^{g}=\mbox{tr}_{M^{i}}q^{L\left(0\right)+h\left(0\right)+\left(h,h\right)/2-c/24}=Z_{M^{i}}\left(h,0,\tau\right).
\]
 Thus the quantum dimension of $\left(M^{i}\right)^{g}$ can
be computed:
\begin{eqnarray*}
\mbox{qdim}\left(M^{i}\right)^{g} & = & \lim_{y\to0}\frac{Z_{i}\left(h,0,iy\right)}{Z_{V}\left(iy\right)}\\
 & = & \lim_{\tau\to i\infty}\frac{Z_{i}\left(h,0,-\frac{1}{\tau}\right)}{Z_{V}\left(-\frac{1}{\tau}\right)}\\
 & = & \lim_{\tau\to i\infty}\frac{\sum_{j}S_{i,j}Z_{j}\left(0,h,\tau\right)}{\sum_{j}S_{0,j}Z_{j}\left(0,h,\tau\right)}\\
 & = & \lim_{q\to0}\frac{\sum_{j}S_{i,j}\mbox{tr}_{M^{j}}e^{2\pi ih\left(0\right)}q^{L\left(0\right)-c/24}}{\sum_{j}S_{0,j}\mbox{tr}_{M^{j}}e^{2\pi ih\left(0\right)}q^{L\left(0\right)-c/24}}\\
 & = & S_{i,0}/S_{0,0}
\end{eqnarray*}
where the last equation follows from the conformal weight $\lambda_{i}>0$
for $0<i\le d$. Remark \ref{quan dim Si0/S00} asserts that
\[
\mbox{qdim}\left(M^{i}\right)^{g}=\mbox{qdim}M^{i}.
\]
  \end{proof}

Let $M$ be an irreducible $\left(V_{\mathbb{Z}\beta}^{+}\right)^{\left\langle \sigma\right\rangle }$-module.
For simplicity, from now on we denote the quantum dimension of $M$
over $\left(V_{\mathbb{Z}\beta}^{+}\right)^{\left\langle \sigma\right\rangle }$
by $\mbox{qdim} M_{\mathbb{}}$ instead of $\mbox{qdim}_{\left(V_{\mathbb{Z}\beta}^{+}\right)^{\left\langle \sigma\right\rangle }}M$.

% ~~~~~~~~~~~~~~~ quantum dimension of twisted module ~~~~~~~~~~~~~~~~~~~

\begin{theorem} \label{quantum dimensions}The quantum dimensions
for all irreducible $\left(V_{\mathbb{Z}\beta}^{+}\right)^{\left\langle \sigma\right\rangle }$-modules
are given by the following tables 1-4.

\begin{center}
\begin{tabular}{|c|c|c|c|c|c|c|}
\hline
 & $\left(V_{\mathbb{Z}\beta}^{+}\right)^{0}$ & $\left(V_{\mathbb{Z}\beta}^{+}\right)^{1}$ & $\left(V_{\mathbb{Z}\beta}^{+}\right)^{2}$  & $V_{\mathbb{Z}\beta}^{-}$  & $V_{\mathbb{Z}\beta+\frac{1}{8}\beta}$ & $V_{\mathbb{Z}\beta+\frac{3}{8}\beta}$\tabularnewline
\hline
$\text{\ensuremath{\omega}}$ & 0 & 4 & 4 & 1 & $\frac{1}{16}$ & $\frac{9}{16}$\tabularnewline
\hline
$\mbox{qdim}$ & $1$ & 1 & 1 & 3 & 6 & 6\tabularnewline
\hline
\end{tabular}
\par\end{center}

\begin{center}
\begin{tabular}{|c|c|c|c|c|c|c|}
\hline
 & $W_{\sigma,1}^{0}$ & $W_{\sigma,1}^{1}$ & $W_{\sigma,1}^{2}$ & $W_{\sigma,2}^{0}$ & $W_{\sigma,2}^{1}$ & $W_{\sigma,2}^{2}$\tabularnewline
\hline
$\omega$ & $\frac{1}{36}$ & $\frac{25}{36}$ & $\frac{49}{36}$ & $\frac{1}{9}$ & $\frac{4}{9}$ & $\frac{16}{9}$\tabularnewline
\hline
$\mbox{qdim}$ & 4 & 4 & 4 & 4 & 4 & 4\tabularnewline
\hline
\end{tabular}
\par\end{center}

\begin{center}
\begin{tabular}{|c|c|c|c|c|c|c|}
\hline
 & $W_{\sigma^{2},1}^{0}$ & $W_{\sigma^{2},1}^{1}$ & $W_{\sigma^{2},1}^{2}$ & $W_{\sigma^{2},2}^{0}$ & $W_{\sigma^{2},2}^{1}$ & $W_{\sigma^{2},2}^{2}$\tabularnewline
\hline
$\omega$ & $\frac{1}{36}$ & $\frac{25}{36}$ & $\frac{49}{36}$ & $\frac{1}{9}$ & $\frac{4}{9}$ & $\frac{16}{9}$\tabularnewline
\hline
$\mbox{qdim}$ & 4 & 4 & 4 & 4 & 4 & 4\tabularnewline
\hline
\end{tabular}
\par\end{center}

\begin{center}
\begin{tabular}{|c|c|c|c|}
\hline
 & $V_{\mathbb{Z}\beta+\frac{1}{4}\beta}^{0}$ & $V_{\mathbb{Z}\beta+\frac{1}{4}\beta}^{1}$ & $V_{\mathbb{Z}\beta+\frac{1}{4}\beta}^{2}$\tabularnewline
\hline
$\omega$ & $\frac{1}{4}$ & $\frac{9}{4}$ & $\frac{9}{4}$\tabularnewline
\hline
$\mbox{qdim}$ & 2 & 2 & 2\tabularnewline
\hline
\end{tabular}
\par\end{center}

\end{theorem}

\begin{proof} 1) We know from Remark \ref{A4 12dim} that
\[
\mbox{qdim}V_{L_{2}}=\left|A_{4}\right|=12.
\]
 It is also obvious that
\[
V_{L_{2}}=V_{\mathbb{Z}\beta}^{+}\oplus V_{\mathbb{Z}\beta}^{-}\oplus V_{\mathbb{Z}\beta+\frac{\beta}{2}}^{+}\oplus V_{\mathbb{Z}\beta+\frac{\beta}{2}}^{-},
\]
 where as $V_{L_{2}}^{A_{4}}$-module $V_{\mathbb{Z}\beta}^{+}=\left(V_{\mathbb{Z}\beta}^{+}\right)^{0}\oplus\left(V_{\mathbb{Z}\beta}^{+}\right)^{1}\oplus\left(V_{\mathbb{Z}\beta}^{+}\right)^{2}$
and $V_{\mathbb{Z}\beta}^{-}\cong V_{\mathbb{Z}\beta+\frac{\beta}{2}}^{+}\cong V_{\mathbb{Z}\beta+\frac{\beta}{2}}^{-}.$
Since $ $$V_{\mathbb{Z}\beta}^{-},\ V_{\mathbb{Z}\beta+\frac{\beta}{2}}^{+},\ V_{\mathbb{Z}\beta+\frac{\beta}{2}}^{-}$
are all simple currents of $V_{\mathbb{Z}\beta}^{+},$ one gets
\[
\mbox{qdim}V_{\mathbb{Z}\beta}^{-}=\mbox{qdim}V_{\mathbb{Z}\beta}^{+}=3.
\]
From (\ref{Decomp1}), we have the decomposition as irreducible $\left(V_{\mathbb{Z}\beta}^{+}\right)^{\left\langle \sigma\right\rangle }$-modules:

\[
V_{\mathbb{Z}\beta}^{+}=\left(V_{\mathbb{Z}\beta}^{+}\right)^{0}\oplus\left(V_{\mathbb{Z}\beta}^{+}\right)^{1}\oplus\left(V_{\mathbb{Z}\beta}^{+}\right)^{2}
\]
where $\left(V_{\mathbb{Z}\beta}^{+}\right)^{0}=\left(V_{\mathbb{Z}\beta}^{+}\right)^{\left\langle \sigma\right\rangle }$
with $\mbox{qdim}\left(V_{\mathbb{Z}\beta}^{+}\right)^{\left\langle \sigma\right\rangle }=1$.
Since $\left(V_{\mathbb{Z}\beta}^{+}\right)^{1}$ and $\left(V_{\mathbb{Z}\beta}^{+}\right)^{2}$
are irreducible $\left(V_{\mathbb{Z}\beta}^{+}\right)^{\left\langle \sigma\right\rangle }$-modules,
by Proposition \ref{quantum-product} and 1),
we get

\[
\mbox{qdim}\left(V_{\mathbb{Z}\beta}^{+}\right)^{1}=\mbox{qdim}\left(V_{\mathbb{Z}\beta}^{+}\right)^{2}=1.
\]

2) By \cite{DL}, every irreducible $V_{L}$-module is a simple current.
Thus from Proposition \ref{quantum-product} we get
\[
\mbox{qdim}{}_{V_{\mathbb{Z}\beta}}V_{\mathbb{Z}\beta+\frac{1}{8}\beta}=1\ \hbox{and}\ \mbox{qdim}{}_{V_{\mathbb{Z}\beta}^{+}}V_{\mathbb{Z}\beta+\frac{1}{8}\beta}=2.
\]
Hence
\[
\mbox{qdim}V_{\mathbb{Z}\beta+\frac{1}{8}\beta}=\mbox{qdim}V_{\mathbb{Z}\beta}=6.
\]

3) Recall that $V_{L_{2}}^{\langle\sigma\rangle}=V_{\mathbb{Z}\gamma},$
where $(\gamma,\gamma)=18.$ We get $\mbox{qdim}V_{\mathbb{Z}\gamma}=4,$
since $\mbox{qdim} V_{L_{2}}=12$ and $o(\sigma)=3.$ Notice that by Proposition
\ref{non-twisted decomposition},
\[
V_{\mathbb{Z}\text{\ensuremath{\gamma}}+\frac{1}{6}\gamma}=V_{\mathbb{Z\beta}+\frac{1}{4}\beta}^{0}+V_{\mathbb{Z\beta}+\frac{1}{4}\beta}^{1},
\]
\[
V_{\mathbb{Z}\gamma+\frac{1}{2}\gamma}=V_{\mathbb{Z\beta}+\frac{1}{4}\beta}^{1}+V_{\mathbb{Z\beta}+\frac{1}{4}\beta}^{2},
\]
\[
V_{\mathbb{Z}\gamma-\frac{1}{6}\gamma}=V_{\mathbb{Z\beta}+\frac{1}{4}\beta}^{0}+V_{\mathbb{Z\beta}+\frac{1}{4}\beta}^{2},
\]
 where $q\dim V_{\mathbb{Z}\gamma+\mu}=4,$ for any $\mu=\pm\frac{1}{6}\gamma,\ \frac{1}{2}\gamma$.
It is easy to determine

\[
\mbox{qdim}V_{\mathbb{Z}\beta+\frac{1}{4}\beta}^{0}=\mbox{qdim}V_{\mathbb{Z}\beta+\frac{1}{4}\beta}^{1}=\mbox{qdim}V_{\mathbb{Z}\beta+\frac{1}{4}\beta}^{2}=2.
\]

4) We have $\mbox{qdim}V_{L_{2}}=\mbox{qdim}V_{L_{2}+\frac{\alpha}{2}}=12$.
By Proposition \ref{quantum-product}, we have
\[
\mbox{qdim}_{V}W_{\sigma^{i},j}=12,\ i,j=1,2.
\]
Consider the action of $\left(V_{\mathbb{Z}\beta}^{+}\right)^{k}$
on $W_{\sigma^{i},j}^{k},i,j=1,2;k=0,1,2$. Use the similar argument
in 3), we can prove $\mbox{qdim}W_{\sigma^{i},j}^{k}$ are the same for
all $i,j=1,2;$ $k=0,1,2.$ Therefore we get $\mbox{qdim}W_{\sigma^{i},j}^{k}=4$
for $i,j=1,2;$ $k=0,1,2$.

\end{proof}

\begin{remark}
Let $V$ be a vertex operator algebra with only finitely many irreducible
modules, the global dimension is defined as $\mbox{glob}(V)=\sum_{M\in Irr(V)}\mbox{qdim}(M)^{2}$
\cite{DJX}. Assume $G$ is a finite subgroup of $Aut(G),$ it is
conjectured that $\left|G\right|^{2}\mbox{glob}\left(V\right)=\mbox{glob}V^{G}$,
which was derived in the frame work of conformal nets \cite{X}.
The vertex operator algebra version is still open. However, the quantum
dimensions above verify this conjecture for vertex operator algebra $V_{L_{2}}$
and group $A_{4}.$

 \end{remark}

\section{Fusion rules}
\def\theequation{5.\arabic{equation}}
\setcounter{equation}{0}

In this section, we find fusion rules for irreducible $V_{L_{2}}^{A_{4}}$-modules.
Quantum dimensions play an important role in determining fusions.
We also need the Verlinde formula to deal with these fusion rules that
involve with twisted modules. We first list all fusion products results,
then we give the proof.

Let $W^{1},W^{2},W^{3}$ be irreducible $V_{L_{2}}^{A_{4}}$-modules.
For simplicity, in the following, the space of all intertwining operators
of type $\left(\begin{array}{c}W^3\\
W^{1}\,W^{2}\end{array}\right)$ is denoted by $I\left(\begin{array}{c}W^3\\
W^{1}\,W^{2}\end{array}\right),$
instead of $I_{V_{L_{2}^{A_{4}}}}\left(\begin{array}{c}W^3\\
W^{1}\,W^{2}\end{array}\right).$
The fusion product of $W^{1}$ and $W^{2}$ is denoted by $W^{1}\boxtimes W^{2}$,
instead of $W^{1}\boxtimes_{V_{L_{2}}^{A_{4}}}W^{2}$.

To determine fusion products of $W_{\sigma^{m},i}^{j}\boxtimes W_{\sigma^{m},k}^{l}$,
$i,k,m=1,2$, $j,l=0,1,2$, we first need to find out certain entries
of the $S$-matrix.

\begin{lemma}\label{S-matrix} The entries of the $S$-matrix that
involve with irreducible twisted modules of $V_{L_{2}}^{A_{4}}$ are
given as the table in Appendix.\end{lemma}

\begin{proof} For convenience,  we denote the irreducible $V_{L_{2}}^{A_{4}}$-modules
by $M^{i}$, $i=0,1,\cdots,20$ as following:

\begin{center}
\begin{tabular}{|c|c|c|c|c|c|}
\hline
$\left(V_{\mathbb{Z}\beta}^{+}\right)^{0}$ & $\left(V_{\mathbb{Z}\beta}^{+}\right)^{1}$ & $\left(V_{\mathbb{Z}\beta}^{+}\right)^{2}$  & $V_{\mathbb{Z}\beta}^{-}$  & $V_{\mathbb{Z}\beta+\frac{1}{8}\beta}$ & $V_{\mathbb{Z}\beta+\frac{3}{8}\beta}$\tabularnewline
\hline
$M^{0}$ & $M^{1}$ & $M^{2}$ & $M^{3}$ & $M^{4}$ & $M^{5}$\tabularnewline
\hline
\end{tabular}
\par\end{center}

\begin{center}
\begin{tabular}{|c|c|c|c|c|c|}
\hline
$W_{\sigma,1}^{0}$ & $W_{\sigma,1}^{1}$ & $W_{\sigma,1}^{2}$ & $W_{\sigma,2}^{0}$ & $W_{\sigma,2}^{1}$ & $W_{\sigma,2}^{2}$\tabularnewline
\hline
\hline
$M^{6}$ & $M^{7}$ & $M^{8}$ & $M^{9}$ & $M^{10}$ & $M^{11}$\tabularnewline
\hline
\end{tabular}
\par\end{center}

\begin{center}
\begin{tabular}{|c|c|c|c|c|c|}
\hline
$W_{\sigma^{2},1}^{0}$ & $W_{\sigma^{2},1}^{1}$ & $W_{\sigma^{2},1}^{2}$ & $W_{\sigma^{2},2}^{0}$ & $W_{\sigma^{2},2}^{1}$ & $W_{\sigma^{2},2}^{2}$\tabularnewline
\hline
\hline
$M^{12}$ & $M^{13}$ & $M^{14}$ & $M^{15}$ & $M^{16}$ & $M^{17}$\tabularnewline
\hline
\end{tabular}
\par\end{center}

\begin{center}
\begin{tabular}{|c|c|c|}
\hline
$V_{\mathbb{Z}\beta+\frac{1}{4}\beta}^{0}$ & $V_{\mathbb{Z}\beta+\frac{1}{4}\beta}^{1}$ & $V_{\mathbb{Z}\beta+\frac{1}{4}\beta}^{2}$\tabularnewline
\hline
$M^{18}$ & $M^{19}$ & $M^{20}$\tabularnewline
\hline
\end{tabular}
\par\end{center}

First we consider the vertex operator algebra $V_{\mathbb{Z}\gamma}$
where $\left(\gamma,\gamma\right)=18$. Its irreducible modules
are $V_{\mathbb{Z}\gamma+\lambda_{k}}$, where
$\lambda_{k}=\frac{k}{18}\gamma$ and $k=0,\cdots,17$. By Page 106
\cite{S}, we see that

\[
Z_{V_{\mathbb{Z}\gamma+\lambda_{i}}}\left(-\frac{1}{\tau}\right)=\sum_{j=0}^{17}\frac{1}{\sqrt{18}}e^{-2\pi i\left(\lambda_{i},\lambda_{j}\right)}Z_{V_{\mathbb{Z}\gamma+\lambda_{j}}}\left(\tau\right).
\]
Thus the entries of $S$-matrix $\left(S_{\lambda_{k},\lambda_{l}}\right)$
for $V_{\mathbb{Z}\gamma}$ is given by

\[
S_{\lambda_{k},\lambda_{l}}=\frac{1}{\sqrt{18}}e^{-2\pi i(\lambda_{k},\lambda_{l})},k,l=0,1,\cdots,17.
\]

Denote the $S$-matrix for the vertex operator algebra $V_{L_{2}}^{A_{4}}$
by $\left(S_{i,j}\right)$. By the identification given in (\ref{twisted modules identification})
and $S$-matrix of $V_{\mathbb{Z}\gamma}$, it is easy to see that
$S_{i,0}=\frac{1}{\sqrt{18}},i=6,\cdots,17$. By Remark \ref{quan dim Si0/S00}
and quantum dimensions listed in Theorem \ref{quantum dimensions},
we have
\[
\mbox{qdim}M^{i}=\frac{S_{i,0}}{S_{0,0}}=4,i=6,7,\cdots,17,
\]
 which implies $S_{0,0}=\frac{1}{4\sqrt{18}}$. Hence we have $S_{i,0}=\frac{\mbox{qdim}M^{i}}{4\sqrt{18}}$ for
$i=0,1,\cdots,20$. Applying the quantum dimensions as listed in Theorem
\ref{quantum dimensions}, we get the first column of the table.

Now let $M^{j}\cong V_{\mathbb{Z}\gamma+\lambda_{l}}.$ If $M^{i}$ is not a submodule of $V_{\mathbb{Z}\gamma+\lambda_{k}}$ for all $\lambda_{k},$ then $S_{i,j}=0.$ Otherwise,
 $M^{i}$ is a submodule of $V_{\mathbb{Z}\gamma+\lambda_{k_{s}}}$  for some $k_1,...,k_r$
and $M^{i}$ is not a submodule of $V_{\mathbb{Z}\gamma+\lambda_{p}}$
for all $\lambda_{p}\not=\lambda_{k_{s}},\ \forall s=1,\cdots,r.$ In this case, $S_{ij}=\sum_{s=1}^{r}S_{\lambda_{l},\lambda_{k_{s}}}$.

In this way, we can get the entries of the $S$-matrix as listed in
the table in Appendix. \end{proof}

\begin{theorem}\label{Fusion Rules}The fusion rules for all irreducible
$V_{L_{2}}^{A_{4}}$-modules are given as following (here $\overline{n}$ is remainder when dividing
$n$ by $3$  for $n\in\mathbb{Z}$):

\begin{equation}
\left(V_{\mathbb{Z}\beta}^{+}\right)^{i}\boxtimes\left(V_{\mathbb{Z}\beta}^{+}\right)^{j}=\left(V_{\mathbb{Z}\beta}^{+}\right)^{\overline{i+j}}, \ i,j=0,1,2\label{1.1}
\end{equation}

$ $
\begin{equation}
\left(V_{\mathbb{Z}\beta}^{+}\right)^{i}\boxtimes V_{\mathbb{Z}\beta+\frac{\beta}{4}}^{j}= V_{\mathbb{Z}\beta+\frac{\beta}{4}}^{\overline{i+j}}  \ i,j=0,1,2\label{1.2}
\end{equation}
\begin{equation}
\left(V_{\mathbb{Z}\beta}^{+}\right)^{i}\boxtimes V_{\mathbb{Z}\beta}^{-}=V_{\mathbb{Z}\beta}^{-}, \ i=0,1,2\label{1.3}
\end{equation}

\begin{eqnarray}
\left(V_{_{\mathbb{Z}\beta}}^{+}\right)^{k}\boxtimes W_{\sigma^{i},i}^{l} & = & W_{\sigma^{i},i}^{\overline{l-k}},\ i=1,2;k,l=0,1,2\nonumber \\
\left(V_{\mathbb{Z}\beta}^{+}\right)^{k}\boxtimes W_{\sigma^{i},3-i}^{l} & = & W_{\sigma^{i},3-i}^{\overline{k+l}},\ i=1,2;k,l=0,1,2\label{1.4}
\end{eqnarray}

\begin{equation}
\left(V_{\mathbb{Z}\beta}^{+}\right)^{i}\boxtimes V_{\mathbb{Z}\beta+\frac{j}{8}\beta}=V_{\mathbb{Z}\beta+\frac{j}{8}\beta},\ i=0,1,2;j=1,3\label{1.6-1}
\end{equation}

\begin{equation}
V_{\mathbb{Z}\beta+\frac{\beta}{4}}^{i}\boxtimes V_{\mathbb{Z}\beta}^{-}=V_{\mathbb{Z}\beta+\frac{\beta}{4}}^{0}\oplus V_{\mathbb{Z}\beta+\frac{\beta}{4}}^{1}\oplus V_{\mathbb{Z}\beta+\frac{\beta}{4}}^{2}, \ i=0,1,2\label{2.3}
\end{equation}

\begin{equation}
 V_{\mathbb{Z}\beta+\frac{\beta}{4}}^{i}\boxtimes V_{\mathbb{Z}\beta+\frac{\beta}{4}}^{j}=V_{\mathbb{Z}\beta}^{-}\oplus\left(V_{\mathbb{Z}\beta}^{+}\right)^{\overline{i+j}},\ i,j=0,1,2\label{2.2}
\end{equation}

\begin{eqnarray}
V_{\mathbb{Z}\beta+\frac{\beta}{4}}^{k}\boxtimes W_{\sigma^{i},3-i}^{l} & = & W_{\sigma^{i},i}^{\overline{2l-k}}\oplus W_{\sigma^{i},i}^{\overline{2l-k+1}}, \ i=1,2;k,l=0,1,2\nonumber \\
V_{\mathbb{Z}\beta+\frac{\beta}{4}}^{k}\boxtimes W_{\sigma^{i},i}^{l} & = & W_{\sigma^{i},3-i}^{\overline{k-l}}\oplus W_{\sigma^{i},3-i}^{\overline{k-l+1}}, \ i=1,2;k,l=0,1,2\label{2.4}
\end{eqnarray}

\begin{equation}
 V_{\mathbb{Z}\beta+\frac{\beta}{4}}^{k}\boxtimes V_{\mathbb{Z}\beta+\frac{l}{8}\beta}=V_{\mathbb{Z}\beta+\frac{1}{8}\beta}\oplus V_{\mathbb{Z}\beta+\frac{3}{8}\beta}, \ k=0,1,2;l=1,3\label{2.6}
\end{equation}

\begin{equation}
V_{\mathbb{Z}\beta}^{-}\boxtimes V_{\mathbb{Z}\beta}^{-}=
\left(V_{\mathbb{Z}\beta}^{+}\right)^{0}\oplus\left(V_{\mathbb{Z}\beta}^{+}\right)^{1}\oplus\left(V_{\mathbb{Z}\beta}^{+}\right)^{2}\oplus2V_{\mathbb{Z}\beta}^{-}\label{3.3}
\end{equation}

\begin{equation}
V_{\mathbb{Z}\beta}^{-}\boxtimes W_{\sigma^{i},j}^{k}=W_{\sigma^{i},j}^{0}\oplus W_{\sigma^{i},j}^{1}\oplus W_{\sigma^{i},j}^{2}, \ i,j=1,2;k=0,1,2\label{3.4}
\end{equation}

\begin{eqnarray}
V_{\mathbb{Z}\beta}^{-}\boxtimes V_{\mathbb{Z}\beta+\frac{1}{8}\beta} & = & V_{\mathbb{Z}\beta+\frac{1}{8}\beta}\oplus2V_{\mathbb{Z}\beta+\frac{3}{8}\beta}\nonumber \\
V_{\mathbb{Z}\beta}^{-}\boxtimes V_{\mathbb{Z}\beta+\frac{3}{8}\beta} & = & 2V_{\mathbb{Z}\beta+\frac{1}{8}\beta}\oplus V_{\mathbb{Z}\beta+\frac{3}{8}\beta}\label{3.6}
\end{eqnarray}

\begin{eqnarray}
V_{\mathbb{Z}\beta+\frac{r}{8}\beta}\boxtimes V_{\mathbb{Z}\beta+\frac{r}{8}\beta} & = & V_{\mathbb{Z}\beta+\frac{\beta}{4}}^{0}\oplus V_{\mathbb{Z}\beta+\frac{\beta}{4}}^{1}\oplus V_{\mathbb{Z}\beta+\frac{\beta}{4}}^{2}\nonumber \\
 &  & \oplus\left(V_{\mathbb{Z}\beta}^{+}\right)^{0}\oplus\left(V_{\mathbb{Z}\beta}^{+}\right)^{1}\oplus\left(V_{\mathbb{Z}\beta}^{+}\right)^{2}\nonumber \\
 &  & \oplus V_{\mathbb{Z}\beta}^{-}\oplus2V_{\mathbb{Z}\beta+\frac{\beta}{8}}\oplus2V_{\mathbb{Z}\beta+\frac{3\beta}{8}},\ \ r=1,3.\label{6.6-1}
\end{eqnarray}

\begin{eqnarray}
V_{\mathbb{Z}\beta+\frac{1}{8}\beta}\boxtimes V_{\mathbb{Z}\beta+\frac{3}{8}\beta} & = & V_{\mathbb{Z}\beta+\frac{\beta}{4}}^{0}\oplus V_{\mathbb{Z}\beta+\frac{\beta}{4}}^{1}\oplus V_{\mathbb{Z}\beta+\frac{\beta}{4}}^{2}\nonumber \\
 &  & \oplus2V_{\mathbb{Z}\beta}^{-}\oplus2V_{\mathbb{Z}\beta+\frac{\beta}{8}}\oplus2V_{\mathbb{Z}\beta+\frac{3\beta}{8}}\label{6.6-2}
\end{eqnarray}

\begin{equation}
W_{\sigma^{i},1}^{k}\boxtimes W_{\sigma^{i},1}^{l}=W_{\sigma^{3-i},1}^{0}\text{\ensuremath{\oplus}}W_{\sigma^{3-i},1}^{1}\oplus W_{\sigma^{3-i},1}^{2}\oplus W_{\sigma^{3-i},2}^{\overline{-i(k+l)}}\label{4.4-3}
\end{equation}

\begin{equation}
W_{\sigma^{i},2}^{k}\boxtimes W_{\sigma^{i},2}^{l}=W_{\sigma^{3-i},1}^{0}\text{\ensuremath{\oplus}}W_{\sigma^{3-i},1}^{1}\oplus W_{\sigma^{3-i},1}^{2}\oplus W_{\sigma^{3-i},2}^{\overline{1+i(k+l)}}\label{4.4-4}
\end{equation}

\begin{equation}
W_{\sigma^{i},1}^{k}\boxtimes W_{\sigma^{i},2}^{l}=\oplus_{k=0}^{2}W_{\sigma^{3-i},2}^{k}\oplus W_{\sigma^{3-i},2}^{\overline{i(l-k)}}\label{4.4-5}
\end{equation}

\begin{eqnarray}
W_{\sigma^{i},j}^{k}\boxtimes V_{\mathbb{Z}\beta+\frac{s}{8}\beta} & = & W_{\sigma^{i},1}^{0}\oplus W_{\sigma^{i},1}^{1}\oplus W_{\sigma^{i},1}^{2}\oplus W_{\sigma^{i},2}^{0}\oplus W_{\sigma^{i},2}^{1}\oplus W_{\sigma^{i},2}^{2}\label{4.6}\\
 &  & i,j=1,2;k=0,1,2;s=1,3.\nonumber
\end{eqnarray}

For $k,l=0,1,2$, $r,i=1,2$,

\begin{equation}
W_{\sigma,r}^{k}\boxtimes W_{\sigma^{2,},r}^{l}=(V_{\mathbb{Z}\beta}^{+})^{\overline{r(l-k)}}\oplus V_{\mathbb{Z}\beta}^{-}\oplus V_{\mathbb{Z}\beta+\frac{1}{8}\beta}\oplus V_{\mathbb{Z}\beta+\frac{3}{8}\beta}\label{4.4-1}
\end{equation}

\[
\]
\begin{equation}
W_{\sigma,r}^{k}\boxtimes W_{\sigma^{2,},3-r}^{l}=(V_{\mathbb{Z}\beta+\frac{1}{4}\beta})^{\overline{r(-k-l)}}\oplus(V_{\mathbb{Z}\beta+\frac{1}{4}\beta})^{\overline{r(-k\text{-}l+1)}}\oplus V_{\mathbb{Z}\beta+\frac{1}{8}\beta}\oplus V_{\mathbb{Z}\beta+\frac{3}{8}\beta}\label{4.4-2}
\end{equation}

\end{theorem}

\begin{proof}

\emph{(\ref{1.1}), (\ref{1.2}), (\ref{1.4}) and (\ref{2.4}) }are
obvious by Proposition \ref{non-twisted decomposition} and fusion
rules for irreducible $V_{\mathbb{Z}\gamma}$-modules.

\emph{Proof of (\ref{1.3}): } By Proposition \ref{quantum-product}, each irreducible
module with quantum dimension 1 is a simple current. Thus the right
hand side should be an irreducible module with quantum dimension $3$ while
$V_{\mathbb{Z}\beta}^{-}$ is the only irreducible module with such
quantum dimension.

\emph{Proof of (\ref{1.6-1}):} By fusion rules for irreducible $V_{\mathbb{Z}\beta}^{+}$-modules,
$I_{V_{\mathbb{Z}\beta}^{+}}\left(_{V_{\mathbb{Z}\beta}^{+}\ V_{\mathbb{Z}\beta+\frac{j}{8}\beta}}^{\ \ V_{\mathbb{Z}\beta+\frac{j}{8}\beta}}\right)\not=0$.
Since $\left(V_{\mathbb{Z}\beta}^{+}\right)^{i}\subset V_{\mathbb{Z}\beta}^{+}$
is a simple current of $\left(V_{\mathbb{Z}\beta}^{+}\right)^{0},$
we get the desired fusion rule.

\emph{Proof of (\ref{2.3}):} First by fusion rules for $V_{\mathbb{Z}\beta}^{+}$-modules, we have
\[
V_{\mathbb{Z}\beta}^{-}\boxtimes_{V_{\mathbb{Z}\beta}^+}V_{\mathbb{Z}\beta+\frac{\beta}{4}}=V_{\mathbb{Z}\beta+\frac{\beta}{4}}
\]
Let $\mathcal{Y}\left(\cdot,z\right)$ be the intertwining operator
of type $\left(_{V_{\mathbb{Z}\beta}^{-}\ V_{\mathbb{Z}\beta+\frac{\beta}{4}}}^{\ \ V_{\mathbb{Z}\beta+\frac{\beta}{4}}}\right).$
For a fixed $v\in V_{\mathbb{Z}\beta+\frac{\beta}{4}}^{i}$,
consider $\mathcal{Y}\left(u,z\right)v,\ u\in V_{\mathbb{Z}\beta}^{-}$.
Then $\left\langle u_{i}v|u\in V_{\mathbb{Z}\beta}^{-},i\in\Z\right\rangle=V_{\mathbb{Z}\beta+\frac{\beta}{4}}^{\text{+}}.$
 Thus we get fusion product
for irreducible $V_{L_{2}}^{A_{4}}$-modules as follows:

\[
V_{\mathbb{Z}\beta}^{-}\boxtimes V_{\mathbb{Z}\beta+\frac{\beta}{4}}^{i}= V_{\mathbb{Z}\beta+\frac{\beta}{4}}^{0}\oplus V_{\mathbb{Z}\beta+\frac{\beta}{4}}^{1}\oplus V_{\mathbb{Z}\beta+\frac{\beta}{4}}^{2},i=0,1,2.
\]

\emph{Proof of (\ref{2.2}): } From (\ref{2.3}), we see that $I\left(_{ V_{\mathbb{Z}\beta+\frac{\beta}{4}}^{i}\ V_{\mathbb{Z}\beta}^{-}}^{\ \  V_{\mathbb{Z}\beta+\frac{\beta}{4}}^{j}}\right)\not=0$,
$i,j=0,1,2.$ $ $Since $ $$\left( V_{\mathbb{Z}\beta+\frac{\beta}{4}}^{1}\right)^{'}= V_{\mathbb{Z}\beta+\frac{\beta}{4}}^{2}$,
$\left(V_{\mathbb{Z}\beta}^{-}\right)^{'}=V_{\mathbb{Z}\beta}^{-}$
and $\left( V_{\mathbb{Z}\beta+\frac{\beta}{4}}^{0}\right)^{'}= V_{\mathbb{Z}\beta+\frac{\beta}{4}}^{0}$.
We obtain
\[
I\left(_{ V_{\mathbb{Z}\beta+\frac{\beta}{4}}^{i}\  V_{\mathbb{Z}\beta+\frac{\beta}{4}}^{j}}^{ \ \ \ \ V_{\mathbb{Z}\beta}^{-}}\right)\not=0,i,j=0,1,2.
\]
By Proposition \ref{fusion rule symmmetry property}, $I\left(_{ V_{\mathbb{Z}\beta+\frac{\beta}{4}}^{i} V_{\mathbb{Z}\beta+\frac{\beta}{4}}^{j}}^{\ \ \left(V_{\mathbb{Z}\beta}^{+}\right)^{\overline{i+j}}}\right)\not=0$.
By counting quantum dimensions, we get (\ref{2.2}).

\emph{Proof of (\ref{2.6}):} Since $ V_{\mathbb{Z}\beta+\frac{\beta}{4}}^{i}\subset V_{\mathbb{Z}\beta+\frac{\beta}{4}}$
and $V_{\mathbb{Z}\beta+\frac{\beta}{4}}$ is an irreducible $V_{\mathbb{Z}\beta}^{+}$-module.
By fusion rules of irreducible $V_{\mathbb{Z}\beta}^{+}$-modules,
we get $I\left(_{ V_{\mathbb{Z}\beta+\frac{\beta}{4}}^{i}V_{\mathbb{Z}\beta+\frac{1}{8}\beta}}^{\ \ V_{\mathbb{Z}\beta+\frac{r}{8}\beta}}\right)\not=0,$
$r=1,3.$ By counting quantum dimensions of both sides, we get the
desired fusion product.

\emph{Proof of (\ref{3.3}):} First by (\ref{1.3}) and Proposition \ref{fusion rule symmmetry property},
we get $I\left(_{V_{\mathbb{Z}\beta}^{-}\ V_{\mathbb{Z}\beta}^{-}}^{\left(V_{\mathbb{Z}\beta}^{+}\right)^{i}}\right)\not=0$,
$i=0,1,2$. By fusion rules for irreducible $V_{L}^{+}$-modules \cite{A2}, we
get $I_{V_{\mathbb{Z}\beta}^{+}}\left(_{V_{\mathbb{Z}\beta}^{-}\ V_{\mathbb{Z}\beta+\frac{1}{2}\beta}^{\pm}}^{\ \ V_{\mathbb{Z}\beta+\frac{1}{2}\beta}^{\mp}}\right)\not=0$.
Using the isomorphism
\[
V_{\mathbb{Z}\beta}^{-}\cong V_{\mathbb{Z}\beta+\frac{1}{2}\beta}^{+}\cong V_{\mathbb{Z}\beta+\frac{1}{2}\beta}^{-},
\]
gives $I\left(_{V_{\mathbb{Z}\beta}^{-}\ V_{\mathbb{Z}\beta}^{-}}^{\ \ V_{\mathbb{Z}\beta}^{-}}\right)\neq0.$
 So it suffices to prove that $I\left(_{V_{\mathbb{Z}\beta}^{-}\ V_{\mathbb{Z}\beta}^{-}}^{\ \ V_{\mathbb{Z}\beta}^{-}}\right)=2.$
 Let $\mathcal{Y}_{1}\left(\cdot,z\right)$,
$\mathcal{Y}_{2}\left(\cdot,z\right)$ be the standard intertwining operators of types
$\left(_{V_{\mathbb{Z}\beta}^{-}\,V_{\mathbb{Z}\beta+\frac{\beta}{2}}^{+}}^{\ \ V_{\mathbb{Z}\beta}^{-}}\right)$
and $\left(_{V_{\mathbb{Z}\beta}^{-}\,V_{\mathbb{Z}\beta+\frac{\beta}{2}}^{-}}^{\ \ V_{\mathbb{Z}\beta}^{-}}\right)$
respectively (see \cite{DL}, \cite{A2}). Note that $e^{\beta}-e^{-\beta}\in V_{\mathbb{Z}\beta}^{-}$, $e^{\frac{\beta}{2}}+e^{-\frac{\beta}{2}}\in V_{\mathbb{Z}\beta+\frac{\beta}{2}}^{+}$
and $e^{\frac{\beta}{2}}-e^{-\frac{\beta}{2}}\in V_{\mathbb{Z}\beta+\frac{\beta}{2}}^{-}$.
Then we have
\begin{eqnarray*}
& &\ \ \ \  \mathcal{Y}_{1}\left(\text{\ensuremath{\beta}(-1)\textbf{1}},z\right)\left(e^{\frac{\beta}{2}}
+e^{-\frac{\beta}{2}}\right)\\
& &=4\left(
e^{\frac{\beta}{2}}-e^{-\frac{\beta}{2}}\right)z^{-1}
+\beta\left(-1\right)\left(e^{\frac{\beta}{2}}+e^{-\frac{\beta}{2}}\right)z^{0}
+{\rm higher\ power\ terms\ of }\ z,
\end{eqnarray*}
\begin{eqnarray*}
& &\ \ \ \  \mathcal{Y}_{1}\left(\text{\ensuremath{\beta}(-1)\textbf{1}},z\right)\left(e^{\frac{\beta}{2}}
-e^{-\frac{\beta}{2}}\right)\\
& &=4\left(
e^{\frac{\beta}{2}}+e^{-\frac{\beta}{2}}\right)z^{-1}
+\beta\left(-1\right)\left(e^{\frac{\beta}{2}}-e^{-\frac{\beta}{2}}\right)z^{0}
+{\rm higher\ power\ terms\ of }\ z.
\end{eqnarray*}
We also have
\begin{eqnarray*}
 &  &\ \ \ \  \mathcal{Y}_{1}\left(e^{\beta}-e^{-\beta},z\right)\left(e^{\frac{\beta}{2}}+e^{-\frac{\beta}{2}}\right)\\
 & &= Y\left(e^{\beta},z\right)\left(e^{\frac{\beta}{2}}+e^{-\frac{\beta}{2}}\right)-Y\left(e^{-\beta},z\right)\left(e^{\frac{\beta}{2}}+e^{-\frac{\beta}{2}}\right)\\
 & &= E^{-}\left(-\beta,z\right)E^{+}\left(-\beta,z\right)e^{\beta}z^{\beta}\left(e^{\frac{\beta}{2}}+e^{-\frac{\beta}{2}}\right)\\
 &  & \ \ \ \ -E^{-}\left(\beta,z\right)E^{+}\left(\beta,z\right)e^{-\beta}z^{-\beta}\left(e^{\frac{\beta}{2}}+e^{-\frac{\beta}{2}}\right)\\
 & & = E^{-}\left(-\beta,z\right)\left(z^4e^{\frac{3\beta}{2}}+z^{-4}e^{\frac{\beta}{2}}\right)-E^{-}\left(\beta,z\right)\left(z^{-4}e^{-\frac{\beta}{2}}+z^4e^{-\frac{3\beta}{2}}\right)\\
 & &= \exp\left(\sum_{n<0}\frac{-\beta\left(n\right)}{n}z^{n}\right)\left(z^4e^{\frac{3\beta}{2}}+z^{-4}e^{\frac{\beta}{2}}\right)\\
 &  & \ \ \ \ -\exp\left(\sum_{n<0}\frac{\beta\left(n\right)}{n}z^{n}\right)\left(z^{-4}e^{-\frac{\beta}{2}}+z^4e^{-\frac{3\beta}{2}}\right)\\
 & &=\beta\left(-1\right)\left(e^{\frac{\beta}{2}}+e^{-\frac{\beta}{2}}\right)z^{-4}+{\rm higher\ power\ terms\ of }\ z,
\end{eqnarray*}
\begin{eqnarray*}
 &  & \ \ \ \ \mathcal{Y}_{2}\left(e^{\beta}-e^{-\beta},z\right)\left(e^{\frac{\beta}{2}}-e^{-\frac{\beta}{2}}\right)\\
 &  &= Y\left(e^{\beta},z\right)\left(e^{\frac{\beta}{2}}-e^{-\frac{\beta}{2}}\right)-Y\left(e^{-\beta},z\right)\left(e^{\frac{\beta}{2}}-e^{-\frac{\beta}{2}}\right)\\
 & & =E^{-}\left(-\beta,z\right)E^{+}\left(-\beta,z\right)e^{\beta}z^{\beta}\left(e^{\frac{\beta}{2}}-e^{-\frac{\beta}{2}}\right)\\
 &  & \ \ \ \ -E^{-}\left(\beta,z\right)E^{+}\left(\beta,z\right)e^{-\beta}z^{-\beta}\left(e^{\frac{\beta}{2}}-e^{-\frac{\beta}{2}}\right)\\
 & &= E^{-}\left(-\beta,z\right)\left(z^4e^{\frac{3\beta}{2}}-z^{-4}e^{\frac{\beta}{2}}\right)-E^{-}\left(\beta,z\right)\left(z^{-4}e^{-\frac{\beta}{2}}-z^4e^{-\frac{3\beta}{2}}\right)\\
 & &= \exp\left(\sum_{n<0}\frac{-\beta\left(n\right)}{n}z^{n}\right)\left(z^4e^{\frac{3\beta}{2}}-z^{-4}e^{\frac{\beta}{2}}\right)\\
 & & \ \ \ \ -\exp\left(\sum_{n<0}\frac{\beta\left(n\right)}{n}z^{n}\right)\left(z^{-4}e^{-\frac{\beta}{2}}-z^{4}e^{-\frac{3\beta}{2}}\right)\\
 & &=-\beta\left(-1\right)\left(e^{\frac{\beta}{2}}-e^{-\frac{\beta}{2}}\right)z^{-4}+{\rm higher\ power\ terms\ of }\ z.
\end{eqnarray*}

From the above computations we see immediately that $\mathcal{Y}_{1}\left(\cdot,z\right)$
and $\mathcal{Y}_{2}\left(\cdot,z\right)$ are linearly independent.
Thus we obtain $N^{V_{\mathbb{Z}\beta}^{-}}_{V_{\mathbb{Z}\beta}^{-}\,V_{\mathbb{Z}\beta}^{-}}\ge2$.
By counting quantum dimensions as listed in Theorem \ref{quantum dimensions},
we get $N^{V_{\mathbb{Z}\beta}^{-}}_{V_{\mathbb{Z}\beta}^{-}\,V_{\mathbb{Z}\beta}^{-}}=2$
and hence we proved (\ref{3.3}).

\emph{Proof of (\ref{3.4}):} This is clear by fusion rules for irreducible $V_{\mathbb{Z}\gamma}$-modules
and the identification in Proposition \ref{non-twisted decomposition}
and \ref{twisted modules identification}.

\emph{Proof of (\ref{3.6}): } Since
\[
V_{\mathbb{Z}\beta+\frac{\beta}{8}}\cong V_{\mathbb{Z}\beta}^{T_{2},+}\cong V_{\mathbb{Z}\beta}^{T_{1},+},\ V_{\mathbb{Z}\beta+\frac{3\beta}{8}}\cong V_{\mathbb{Z}\beta}^{T_{2},-}\cong V_{\mathbb{Z}\beta}^{T_{1},-},\ V_{\mathbb{Z}\beta}^{-}\cong V_{\mathbb{Z}\beta+\frac{\beta}{2}}^{-}\cong V_{\mathbb{Z}\beta+\frac{\beta}{2}}^{+}
\]
as irreducible $V_{L_{2}}^{A_{4}}$-modules \cite{DJ4}, it follows from the fusion
rules of irreducible $V_{\mathbb{Z}\beta}^{+}$-modules \cite{A2} that $
I\left(_{V_{\mathbb{Z}\beta}^{-}\ V_{\mathbb{Z}\beta+\frac{j}{8}\beta}}^{\ \  V_{\mathbb{Z}\beta+\frac{i}{8}\beta}}\right)\not=0,\ i,j=1,3,i\not=j.
$
It suffices to prove that $ $$I\left(_{V_{\mathbb{Z}\beta}^{-}\ V_{\mathbb{Z}\beta+\frac{j}{8}\beta}}^{\ \ V_{\mathbb{Z}\beta+\frac{i}{8}\beta}}\right)=2$
for $i,j=1,3,i\not=j.$ First we prove $I\left(_{V_{\mathbb{Z}\beta}^{-}\ V_{\mathbb{Z}\beta+\frac{1}{8}\beta}}^{\ \ V_{\mathbb{Z}\beta+\frac{3}{8}\beta}}\right)=2.$
Let $\mathcal{Y}_{1}\left(\cdot,z\right)\in I_{V_{\mathbb{Z}\beta}^{+}}\left(_{V_{\mathbb{Z}\beta+\frac{\beta}{2}}^{+}\ V_{\mathbb{Z}\beta+\frac{\beta}{8}}}^{\ \ V_{\mathbb{Z}\beta+\frac{3\beta}{8}}}\right)$,
$\mathcal{Y}_{2}\left(\cdot,z\right)\in I_{V_{\mathbb{Z}\beta}^{+}}\left(_{V_{\mathbb{Z}\beta+\frac{\beta}{2}}^{-}\ V_{\mathbb{Z}\beta+\frac{\beta}{8}}}^{\ \ V_{\mathbb{Z}\beta+\frac{3\beta}{8}}}\right).$ Note that
$e^{\frac{\beta}{2}}+e^{-\frac{\beta}{2}}\in V_{\mathbb{Z}\beta+\frac{\beta}{2}}^{+}$,
$e^{\frac{\beta}{2}}-e^{-\frac{\beta}{2}}\in V_{\mathbb{Z}\beta+\frac{\beta}{2}}^{-}$,
$e^{\frac{\beta}{8}},e^{-\frac{7\beta}{8}}\in V_{\mathbb{Z}\beta+\frac{\beta}{8}}$.
Considering $\mathcal{Y}_{1}\left(e^{\frac{\beta}{2}}+e^{-\frac{\beta}{2}},z\right)e^{\frac{\beta}{8}},$
$\mathcal{Y}_{2}\left(e^{\frac{\beta}{2}}-e^{-\frac{\beta}{2}},z\right)e^{\frac{\beta}{8}},$
$\mathcal{Y}_{1}\left(e^{\frac{\beta}{2}}+e^{-\frac{\beta}{2}},z\right)e^{-\frac{7\beta}{8}},$ $\mathcal{Y}_{2}\left(e^{\frac{\beta}{2}}-e^{-\frac{\beta}{2}},z\right)e^{-\frac{7\beta}{8}}$
and applying similar argument as in the proof of (\ref{3.3}), we
can prove $N_{V_{\mathbb{Z}\beta}^{-}\,V_{\mathbb{Z}\beta+\frac{1}{8}\beta}}^{V_{\mathbb{Z}\beta+\frac{3}{8}\beta}}=2.$
Similarly, $N_{V_{\mathbb{Z}\beta}^{-}\, V_{\mathbb{Z}\beta+\frac{3}{8}\beta}}^{V_{\mathbb{Z}\beta+\frac{1}{8}\beta}}=2.$

\emph{Proof of (\ref{6.6-1}): } Case 1: $r=1$. From (\ref{1.6-1}), (\ref{2.6}) and Proposition
\ref{fusion rule symmmetry property}, we get
\[
I\left(_{V_{\mathbb{Z}\beta+\frac{1}{8}\beta}\ V_{\mathbb{Z}\beta+\frac{1}{8}\beta}}^{\ \ V_{\mathbb{Z}\beta+\frac{1}{4}\beta}^{i}}\right)\not=0,\ I\left(_{V_{\mathbb{Z}\beta+\frac{1}{8}\beta}\ V_{\mathbb{Z}\beta+\frac{1}{8}\beta}}^{\ \ \left(V_{\mathbb{Z}\beta}^{+}\right)^{j}}\right)\not=0,i,j=0,1,2.
\]
So it is sufficient to prove $N_{\mathbb{Z}\beta+\frac{1}{8}\beta\ \mathbb{Z}\beta+\frac{1}{8}\beta}^{V_{\mathbb{Z}\beta+\frac{1}{8}\beta}}=2$
and $N_{\mathbb{Z}\beta+\frac{1}{8}\beta\ \mathbb{Z}\beta+\frac{1}{8}\beta}^{V_{\mathbb{Z}\beta+\frac{3}{8}\beta}}=2$
using the quantum dimensions.

Note from \cite{DJ4} that there are isomorphisms of irreducible $V_{L_{2}}^{A_{4}}$-modules:
\[
V_{\mathbb{Z}\beta+\frac{\beta}{8}}\cong V_{\mathbb{Z}\beta}^{T_{1},+}\cong V_{\mathbb{Z}\beta}^{T_{2},+},V_{\mathbb{Z}\beta}^{-}\cong V_{\mathbb{Z}\beta+\frac{\beta}{2}}^{-}\cong V_{\mathbb{Z}\beta+\frac{\beta}{2}}^{+}.
\]
By fusion rules of irreducible $V_{\mathbb{Z}\beta}^{+}$-modules
\cite{A2}, we get
\[
I\left(_{V_{\mathbb{Z}\beta+\frac{1}{8}\beta}\ V_{\mathbb{Z}\beta+\frac{1}{8}\beta}}^{\ \ V_{\mathbb{Z}\beta+\frac{1}{8}\beta}}\right)\not=0,I\left(_{V_{\mathbb{Z}\beta+\frac{1}{8}\beta}\ V_{\mathbb{Z}\beta+\frac{1}{8}\beta}}^{\ \ V_{\mathbb{Z}\beta+\frac{3}{8}\beta}}\right)\not=0.
\]
Let $T=T^{1}\oplus T^{2}$ be the direct sum of irreducible $\mathbb{C}\left[\Z\beta\right]$-modules
$T^{1}$ and $T^{2}$, and define a linear isomorphism $\psi\in\mbox{End}T$
by $\psi\left(t_{1}\right)=t_{2}$, $\psi\left(t_{2}\right)=t_{1}$,
where $t_{i}$ is a basis of $T^{i}$ for $i=1,2$. For $\lambda\in (\Z\beta)^{\circ}$,
we write $\lambda=r\beta/8+m\beta$ for $-3\le r\le 4$ and $m\in\mathbb{Z},$
and define $\psi_{\lambda}\in\mbox{End}T$ by $
\psi_{\lambda}=e_{m\alpha}{\psi}^{r}.$ By fusion rules for irreducible $V_{\Z\beta}^{+}$-modules \cite{A1},
we have $I_{V_{\mathbb{Z}\beta}^{+}}\left(_{V_{\mathbb{Z}\beta+\frac{1}{8}\beta\ V_{\Z\beta}^{T_{j,+}}}}^{\ \ V_{\Z\beta}^{T_{i},+}}\right)\not=0,i,j=1,2$.
Let $\mathcal{Y}_{ij}\left(\cdot,z\right)\in I_{V_{\mathbb{Z}\beta}^{+}}\left(_{V_{\mathbb{Z}\beta+\frac{1}{8}\beta\ V_{\Z\beta}^{T_{j,+}}}}^{\ \ V_{\Z\beta}^{T_{i},+}}\right)$.
Note that the intertwining operator is given by
\[
\mathcal{Y}\left(u,z\right)=\mathcal{Y}^{\theta}\left(u,z\right)\otimes\psi_{\lambda}\ \mbox{for\ }\lambda\in (\Z\beta)^{\circ}\ \mbox{and\ }u\in M\left(1,\lambda\right)
\]
where $\mathcal{Y}^{\theta}\left(e_{\lambda},z\right)=2^{-\left\langle \lambda,\lambda\right\rangle }z^{-\frac{\left\langle \lambda,\lambda\right\rangle }{2}}\mbox{exp}\left(\sum_{n\in1/2+\mathbb{N}}\frac{\lambda\left(-n\right)}{n}z^{n}\right)\mbox{exp}\left(-\sum_{n\in1/2+\mathbb{N}}\frac{\lambda\left(n\right)}{n}z^{n}\right).$

For $1\otimes e_{\beta+\frac{1}{8}\beta}\in V_{\mathbb{Z}\beta+\frac{1}{8}\beta}=M\left(1\right)\otimes\mathbb{C}\left[\mathbb{Z}\beta+\frac{1}{8}\beta\right],$
$t_{1}\in T_{1}$ and $t_{2}\in T_{2}$, we have
\begin{eqnarray}
 &  & \mathcal{Y}_{21}(1\otimes e_{\beta+\frac{1}{8}\beta},z)t_{1}=\mathcal{Y}_{21}^{\theta}(e_{\frac{9}{8}\beta},z)\psi_{\frac{9}{8}\beta}t_{1}\nonumber \\
 &  &\ \ \ \ = \mathcal{Y}_{21}^{\theta}(e_{\frac{9}{8}\beta},z)e_{\beta}\psi_{\frac{1}{8}\beta}t_{1}\nonumber \\
 & & \ \ \ \ =\mathcal{Y}_{21}^{\theta}(e_{\frac{9}{8}\beta},z)e_{\beta}t_{2}\nonumber \\
 &  &\ \ \ \ = -\mathcal{Y}_{21}^{\theta}(e_{\frac{9}{8}\beta},z)t_{2}\nonumber \\
 &  &\ \ \ \ =- 2^{-\frac{81}{8}}z^{-\frac{81}{16}}\mbox{exp}(-\sum_{n\in1/2+\mathbb{N}}\frac{\lambda(n)}{n}z^{n})t_{2},\label{1}
\end{eqnarray}
\begin{eqnarray}
 &  & \mathcal{Y}_{12}(1\otimes e_{\frac{9}{8}\beta},z)t_{2}=\mathcal{Y}_{12}^{\theta}( e_{\frac{9}{8}\beta},z)\psi_{\frac{1}{8}\beta}t_{2}\nonumber \\
 & & \ \ \ \ =\mathcal{Y}_{12}^{\theta}( e_{\frac{9}{8}\beta},z)e_{\beta}\psi t_{2}\nonumber \\
 & &\ \ \ \ = \mathcal{Y}_{12}^{\theta}(e_{\frac{9}{8}\beta},z)e_{\beta}t_{1}\nonumber \\
 & & \ \ \ \ =\mathcal{Y}_{12}^{\theta}( e_{\frac{9}{8}\beta},z)t_{1}\nonumber \\
 &  &\ \ \ \  = 2^{-\frac{81}{8}}z^{-\frac{81}{16}}\mbox{exp}(-\sum_{n\in1/2+\mathbb{N}}\frac{\lambda(n)}{n}z^{n})t_{1}.\label{2}
\end{eqnarray}
We also have
\begin{eqnarray}
 &  & \mathcal{Y}_{21}(1\otimes e_{\frac{1}{8}\beta},z)t_{1}=\mathcal{Y}_{21}^{\theta}(e_{\frac{1}{8}\beta},z)\psi_{\frac{1}{8}\beta}t_{1}\nonumber \\
 & &\ \ \ \ = \mathcal{Y}_{21}^{\theta}(e_{\frac{1}{8}\beta},z)\psi t_{1}\nonumber \\
 & &\ \ \ \ = \mathcal{Y}_{21}^{\theta}(e_{\frac{1}{8}\beta},z)t_{2}\nonumber \\
 & &\ \ \ \ = -\mathcal{Y}_{21}^{\theta}(e_{\frac{1}{8}\beta},z)t_{2}\nonumber \\
 &  &\ \ \ \ = 2^{-\frac{1}{8}}z^{-\frac{1}{16}}\mbox{exp}(-\sum_{n\in1/2+\mathbb{N}}\frac{\lambda(n)}{n}z^{n})t_{2},\label{3}
\end{eqnarray}
\begin{eqnarray}
 &  & \mathcal{Y}_{12}(1\otimes e_{\frac{1}{8}\beta},z)t_{2}=\mathcal{Y}_{12}^{\theta}(e_{\frac{1}{8}\beta},z)\psi_{\frac{1}{8}\beta}t_{2}\nonumber\\
 & &\ \ \ \ =\mathcal{Y}_{12}^{\theta}(e_{\frac{1}{8}\beta},z)\psi t_{2}\nonumber \\
 &  &\ \ \ \ = \mathcal{Y}_{12}^{\theta}(e_{\frac{1}{8}\beta},z) t_{1}\nonumber \\
 &  &\ \ \ \ = 2^{-\frac{1}{8}}z^{-\frac{1}{16}}\mbox{exp}(-\sum_{n\in1/2+\mathbb{N}}\frac{\lambda(n)}{n}z^{n})t_{1}.\label{4}
\end{eqnarray}
 So $\mathcal{Y}_{12}\left(\cdot,z\right), \mathcal{Y}_{21}\left(\cdot,z\right)\in I\left(_{\mathbb{Z}\beta+\frac{1}{8}\beta\ \mathbb{Z}\beta+\frac{1}{8}\beta}^{\ \ V_{\mathbb{Z}\beta+\frac{1}{8}\beta}}\right)$
are linearly independent and
$N_{\mathbb{Z}\beta+\frac{1}{8}\beta\ \mathbb{Z}\beta+\frac{1}{8}\beta}^{ V_{\mathbb{Z}\beta+\frac{1}{8}\beta}}\ge2$.
By a similar argument, we can prove that $N_{\mathbb{Z}\beta+\frac{1}{8}\beta\ \mathbb{Z}\beta+\frac{1}{8}\beta}^{V_{\mathbb{Z}\beta+\frac{3}{8}\beta}}\ge2$.
Counting quantum dimensions of modules in the fusion product then asserts
\[
I\left(_{\mathbb{Z}\beta+\frac{1}{8}\beta\ \mathbb{Z}\beta+\frac{1}{8}\beta}^{\ \ V_{\mathbb{Z}\beta+\frac{1}{8}\beta}}\right)=2\ \mbox{and\ }I\left(_{\mathbb{Z}\beta+\frac{1}{8}\beta\ \mathbb{Z}\beta+\frac{1}{8}\beta}^{\ \ V_{\mathbb{Z}\beta+\frac{3}{8}\beta}}\right)=2.
\]

Case 2: $r=3$. The proof is similar to that of case 1. This finishes the proof of (\ref{6.6-1}).

\emph{Proof of (\ref{6.6-2}):} By (\ref{2.6}), (\ref{3.6}) and Proposition \ref{fusion rule symmmetry property},
we have
\[
I\left(_{V_{\mathbb{Z}\beta+\frac{\beta}{8}}\ V_{\mathbb{Z}\beta+\frac{3\beta}{8}}}^{\ \ V_{\mathbb{Z}\beta+\frac{1}{4}\beta}^{i}}\right)\not=0,i=0,1,2;\ I\left(_{V_{\mathbb{Z}\beta+\frac{\beta}{8}}\ V_{\mathbb{Z}\beta+\frac{3\beta}{8}}}^{\ \ V_{\mathbb{Z}\beta}^{-}}\right)\not=0.
\]
 Note that we have the following isomorphism of irreducible $V_{\mathbb{Z}\alpha}^{A_{4}}$-modules
\cite{DJ4}:
\[
V_{\mathbb{Z}\beta+\frac{\beta}{8}}\cong V_{\mathbb{Z}\beta}^{T_{1},+}\cong V_{\mathbb{Z}\beta}^{T_{2},+},V_{\mathbb{Z}\beta+\frac{3\beta}{8}}\cong V_{\mathbb{Z}\beta}^{T_{1},-}\cong V_{\mathbb{Z}\beta}^{T_{2},-}.
\]
Equation (\ref{6.6-1}) indicates that
\[
I\left(_{\mathbb{Z}\beta+\frac{1}{8}\beta\ \mathbb{Z}\beta+\frac{3}{8}\beta}^{\ \ V_{\mathbb{Z}\beta+\frac{1}{8}\beta}}\right)=2,\ I\left(_{\mathbb{Z}\beta+\frac{1}{8}\beta\ \mathbb{Z}\beta+\frac{3}{8}\beta}^{\ \ V_{\mathbb{Z}\beta+\frac{3}{8}\beta}}\right)=2.
\]

\emph{Proof of (\ref{4.4-3}), (\ref{4.4-4}) and (\ref{4.4-5}):} We can prove these fusion products
by applying Proposition \ref{Verlinde formula} and Lemma \ref{S-matrix}.

\emph{Proof of (\ref{4.6}): } We only give a proof of
\[
W_{\sigma,1}^{0}\boxtimes V_{\mathbb{Z}\beta+\frac{1}{8}\beta}=W_{\sigma,1}^{0}\oplus W_{\sigma,1}^{1}\oplus W_{\sigma,1}^{2}\oplus W_{\sigma,2}^{0}\oplus W_{\sigma,2}^{1}\oplus W_{\sigma,2}^{2}
\]
here and proofs for the other cases are similar.

First we prove that $I\left(_{W_{\sigma,1}^{0}V_{\mathbb{Z}\beta+\frac{1}{8}\beta}}^{\ \ W}\right)=0$
for any irreducible $V_{L_{2}}^{A_{4}}$-module $W$ appearing in
the untwisted $V_{\mathbb{Z}\beta}^{+}$-modules. Otherwise, there
is some $W$ such that $I\left(_{W_{\sigma,1}^{0}V_{\mathbb{Z}\beta+\frac{1}{8}\beta}}^{\ \ W}\right)\not=0$.
By Proposition \ref{fusion rule symmmetry property}, we obtain
$I\left(_{V_{\mathbb{Z}\beta+\frac{1}{8}\beta}\ W'}^{\ \ W_{\sigma^{2},1}^{0}}\right)\not=0$.
The  fusion products $V_{\mathbb{Z}\beta+\frac{1}{8}\beta}\boxtimes W'$
for all such $W$ have been known already. It is easy to see that $I\left(_{V_{\mathbb{Z}\beta+\frac{1}{8}\beta}\ W'}^{\ \ W_{\sigma^{2},1}^{0}}\right)=0$ for all such $W$, which is a contradiction.

Now we show that $I\left(_{W_{\sigma,1}^{0}V_{\mathbb{Z}\beta+\frac{1}{8}\beta}}^{\ \ W_{\sigma^{2},i}^{j}}\right)=0$,
$\forall i=1,2$, $j=0,1,2$. Otherwise, if there exists some $i_{0}\in\{1,2\}$,
$j_{0}\in\left\{ 0,1,2\right\} $ such that $I\left(_{W_{\sigma,1}^{0}V_{\mathbb{Z}\beta+\frac{1}{8}\beta}}^{\ \ W_{\sigma^{2},i_{0}}^{j_{0}}}\right)\not=0$.
Since $\left(V_{\mathbb{Z}\beta+\frac{1}{8}\beta}\right)^{'}=V_{\mathbb{Z}\beta+\frac{1}{8}\beta}$
and $\left(W_{\sigma^{2},i_{0}}^{j_{0}}\right)'=W_{\sigma,i_{0}}^{j_{0}}$, we see that
 $I\left(_{W_{\sigma,1}^{0}W_{\sigma,i_{0}}^{j_{0}}}^{\ \ V_{\mathbb{Z}\beta+\frac{1}{8}\beta}}\right)\not=0$ by Proposition \ref{fusion rule symmmetry property},
which contradicts with (\ref{4.4-3}) or (\ref{4.4-4}).

Thus we have
$W_{\sigma,1}^{0}\boxtimes V_{\mathbb{Z}\beta+\frac{1}{8}\beta}=\oplus_{p,q}m_{p,q}W_{\sigma,p}^{q}$
where $m_{p,q}$ are integers. Assume that $m_{p,q}\not=0$ for some
$p\in\left\{ 1,2\right\} $, $q\in\left\{ 0,1,2\right\} $, then by
(\ref{1.4}) and (\ref{1.6-1}) we have $m_{p0}=m_{p1}=m_{p2}\not=0$. Assume that $m_{3-p,k}=0$ for all $k=0,1,2$. Then by quantum dimensions
of each module, we get
\begin{equation}
W_{\sigma,1}^{0}\boxtimes V_{\mathbb{Z}\beta+\frac{1}{8}\beta}=2W_{\sigma,p}^{0}\oplus2W_{\sigma,p}^{1}\oplus2W_{\sigma,p}^{2}.\label{assumption 4.6}
\end{equation}
By (\ref{assumption 4.6}) and (\ref{2.4}) we obtain
\begin{equation}
\left(V_{\mathbb{Z}\beta+\frac{1}{4}\beta}\right)^{0}\boxtimes\left(W_{\sigma,1}^{0}\boxtimes V_{\mathbb{Z}\beta+\frac{1}{8}\beta}\right)=4W_{\sigma,2}^{0}\oplus4W_{\sigma,2}^{1}\oplus4W_{\sigma,2}^{2}.\label{assumption 4.6 proof}
\end{equation}
But by associativity of fusion product and (\ref{2.6}) we have
\begin{eqnarray*}
& & V_{\mathbb{Z}\beta+\frac{\beta}{4}}^{0}\boxtimes\left(V_{\mathbb{Z}\beta+\frac{1}{8}\beta}\boxtimes W_{\sigma,1}^{0}\right) =  \left(V_{\mathbb{Z}\beta+\frac{\beta}{4}}^{0}\boxtimes V_{\mathbb{Z}\beta+\frac{1}{8}\beta}\right)\boxtimes W_{\sigma,1}^{0}\\
 & & \ \ \ \ =\left(V_{\mathbb{Z}\beta+\frac{1}{8}\beta}\oplus V_{\mathbb{Z}\beta+\frac{3}{8}\beta}\right)\boxtimes W_{\sigma,1}^{0}\\
 & &\ \ \ \ = 2W_{\sigma,p}^{0}\oplus2W_{\sigma,p}^{1}\oplus2W_{\sigma,p}^{2}\oplus V_{\mathbb{Z}\beta+\frac{3}{8}\beta}\boxtimes W_{\sigma,1}^{0},
\end{eqnarray*}
a contradiction  with (\ref{assumption 4.6 proof}). Hence there
exists some $l=0,1,2$ such that $m_{3-p,l}\not=0$, then we also
have $m_{3-p,0}=m_{3-p,1}=m_{3-p,2}\not=0$ by applying (\ref{1.4}).
By counting quantum dimensions of both sides, we see that
\[
W_{\sigma,1}^{0}\boxtimes V_{\mathbb{Z}\beta+\frac{1}{8}\beta}=W_{\sigma,1}^{0}\oplus W_{\sigma,1}^{1}\oplus W_{\sigma,1}^{2}\oplus W_{\sigma,2}^{0}\oplus W_{\sigma,2}^{1}\oplus W_{\sigma,2}^{2}.
\]

\emph{Proof of (\ref{4.4-1}):} Since $\left(W_{\sigma,1}^{0}\right)^{'}=W_{\sigma^{2},1}^{0}$, by
Proposition \ref{fusion rule symmmetry property}, we get
\[
I\left(_{W_{\sigma,1}^{0}\ W_{\sigma^{2},1}^{0}}^{\ \ \left(V_{\mathbb{Z}\beta}^{+}\right)^{0}}\right)\not=0.
\]
By (\ref{3.4}), (\ref{4.6}) and Proposition \ref{fusion rule symmmetry property},
we obtain
\[
I\left(_{W_{\sigma^{2},1}^{0}W_{\sigma,1}^{0}}^{\ \ V_{\mathbb{Z}\beta}^{-}}\right)\not=0,\ I\left(_{W_{\sigma,1}^{0}W_{\sigma^{2},1}^{0}}^{\ \ V_{\mathbb{Z}\beta+\frac{r}{8}}}\right)\not=0,\ r=1,3.
\]
Thus
\begin{equation}
W_{\sigma,1}^{0}\boxtimes W_{\sigma^{2},1}^{0}=\left(V_{\mathbb{Z}\beta}^{+}\right)^{0}\oplus V_{\mathbb{Z}\beta}^{-}\oplus V_{\mathbb{Z}\beta+\frac{1}{8}\beta}\oplus V_{\mathbb{Z}\beta+\frac{3}{8}\beta}\label{4.4-1.1}
\end{equation}
by counting the quantum dimensions.

From (\ref{1.4}), for $k,l=0,1,2$, we have

\[
W_{\sigma,1}^{k}=\left(V_{\mathbb{Z}\beta}^{+}\right)^{\overline{-k}}\boxtimes W_{\sigma,1}^{0},\ W_{\sigma^{2},1}^{l}=\left(V_{\mathbb{Z}\beta}^{+}\right)^{l}\boxtimes W_{\sigma^{2},1}^{0}.
\]
So
\begin{eqnarray}
W_{\sigma,1}^{k}\boxtimes W_{\sigma^{2},1}^{l} & = & \left(V_{\mathbb{Z}\beta}^{+}\right)^{\overline{-k}}\boxtimes\left(V_{\mathbb{Z}\beta}^{+}\right)^{l}\boxtimes\left(W_{\sigma,1}^{0}\boxtimes W_{\sigma^{2},1}^{0}\right)\nonumber \\
 & = & \left(V_{\mathbb{Z}\beta}^{+}\right)^{\overline{l-k}}\boxtimes\left(\left(V_{\mathbb{Z}\beta}^{+}\right)^{0}\oplus V_{\mathbb{Z}\beta}^{-}+V_{\mathbb{Z}\beta+\frac{\beta}{8}}+V_{\mathbb{Z}\beta+\frac{3}{8}\beta}\right)\label{4.4-1.1kl}\\
 & = & \left(V_{\mathbb{Z}\beta}^{+}\right)^{\overline{l-k}}\oplus V_{\mathbb{Z}\beta}^{-}\oplus V_{\mathbb{Z}\beta+\frac{1}{8}\beta}\oplus V_{\mathbb{Z}\beta+\frac{3}{8}\beta}.\nonumber
\end{eqnarray}
Similarly we can prove that
\[
W_{\sigma,2}^{k}\boxtimes W_{\sigma^{2},2}^{l}=\left(V_{\mathbb{Z}\beta}^{+}\right)^{\overline{l-k}}\oplus V_{\mathbb{Z}\beta}^{-}\oplus V_{\mathbb{Z}\beta+\frac{1}{8}\beta}\oplus V_{\mathbb{Z}\beta+\frac{3}{8}\beta}.
\]
This finishes the proof of  (\ref{4.4-1}).

\emph{Proof of (\ref{4.4-2}): } From (\ref{2.4}), we have
\[
I\left(_{W_{\sigma,1}^{0}\ V_{\mathbb{Z}\beta+\frac{\beta}{4}}^{0}}^{\ \ W_{\sigma,2}^{0}}\right)\not=0,\ I\left(_{W_{\sigma,1}^{0}\ V_{\mathbb{Z}\beta+\frac{\beta}{4}}^{2}}^{\ \ W_{\sigma,2}^{0}}\right)\not=0.
\]
Since $\left(W_{\sigma,2}^{0}\right)'=W_{\sigma^{2},2}^{0}$ , $\left( V_{\mathbb{Z}\beta+\frac{\beta}{4}}^{0}\right)^{'}= V_{\mathbb{Z}\beta+\frac{\beta}{4}}^{0}$and
$\left( V_{\mathbb{Z}\beta+\frac{\beta}{4}}^{2}\right)^{'}= V_{\mathbb{Z}\beta+\frac{\beta}{4}}^{1}$,
by Proposition \ref{fusion rule symmmetry property} we obtain
\[
I\left(_{W_{\sigma,1}^{0}\ W_{\sigma^{2},2}^{0}}^{\ \  V_{\mathbb{Z}\beta+\frac{\beta}{4}}^{0}}\right)\not=0,I\left(_{W_{\sigma,1}^{0}\ W_{\sigma^{2},2}^{0}}^{\ \  V_{\mathbb{Z}\beta+\frac{\beta}{4}}^{1}}\right)\not=0.
\]
By (\ref{4.6}), for $r=1,2$, $k=0,1,2$, $s=1,3$,
\[
I\left(_{W_{\sigma,1}^{0}\ V_{\mathbb{Z}\beta+\frac{s}{8}\beta}}^{\ \ W_{\sigma,r}^{k}}\right)\not=0.
\]
Since $\left(V_{\mathbb{Z}\beta+\frac{s}{8}\beta}\right)^{'}=V_{\mathbb{Z}\beta+\frac{s}{8}\beta}$,
we obtain $I\left(_{W_{\sigma,1}^{0}\ W_{\sigma^{2},r}^{k}}^{\ \ V_{\mathbb{Z}\beta+\frac{s}{8}\beta}}\right)\not=0$.
In particular,
\[
I\left(_{W_{\sigma,1}^{0}\ W_{\sigma^{2},r}^{0}}^{\ \ V_{\mathbb{Z}\beta+\frac{s}{8}\beta}}\right)\not=0,s=1,3.
\]
By counting quantum dimensions, we obtain
\begin{equation*}
W_{\sigma,1}^{0}\boxtimes W_{\sigma^{2},2}^{0}= V_{\mathbb{Z}\beta+\frac{\beta}{4}}^{0}\oplus V_{\mathbb{Z}\beta+\frac{\beta}{4}}^{1}\oplus V_{\mathbb{Z}\beta+\frac{\beta}{8}}\oplus V_{\mathbb{Z}\beta+\frac{3}{8}\beta}.\label{4.4-1.2-0}
\end{equation*}

From (\ref{1.4}) we have
\[
W_{\sigma,1}^{k}=\left(V_{\mathbb{Z}\beta}^{+}\right)^{-k}\boxtimes W_{\sigma,1}^{0},\ W_{\sigma^{2},2}^{l}=\left(V_{\mathbb{Z}\beta}^{+}\right)^{-l}\boxtimes W_{\sigma^{2},2}^{0}.
\]
Thus
\begin{eqnarray*}
& &W_{\sigma,1}^{k}\boxtimes W_{\sigma^{2},2}^{l} =  \left(V_{\mathbb{Z}\beta}^{+}\right)^{\overline{-k-l}}\boxtimes\left(W_{\sigma,1}^{0}\boxtimes W_{\sigma^{2},2}^{0}\right)\\
 &  &\ \ \ \ = V_{\mathbb{Z}\beta+\frac{\beta}{4}}^{\overline{-k-l}}\oplus\left(V_{\mathbb{Z}\beta\oplus\frac{\beta}{4}}\right)^{\overline{-k-l+l}}\oplus V_{\mathbb{Z}\beta+\frac{\beta}{8}}\oplus V_{\mathbb{Z}\beta+\frac{3}{8}\beta}.
\end{eqnarray*}

Similarly, we can show
\[
W_{\sigma,2}^{k}\boxtimes W_{\sigma^{2},1}^{l}= V_{\mathbb{Z}\beta+\frac{\beta}{4}}^{\overline{k+l}}\oplus\left(V_{\mathbb{Z}\beta \oplus\frac{\beta}{4}}\right)^{\overline{k+l+l}}\oplus V_{\mathbb{Z}\beta+\frac{\beta}{8}}\oplus V_{\mathbb{Z}\beta+\frac{3}{8}\beta}.
\]
Thus (\ref{4.4-2}) holds.
\end{proof}

\section{Appendix}
\def\theequation{6.\arabic{equation}}
\setcounter{equation}{0}

The following is the part of $S$-matrix for irreducible $V_{L_{2}}^{A_{4}}$-modules
that we need:

\begin{tabular}{|c|c|c|c|c|c|c|c|}
\hline
$\sqrt{18}S_{i,j}$ & 0 & 6 & 7 & 8 & 9 & 10 & 11\tabularnewline
\hline
\hline
0 & $\frac{1}{4}$ & $1$ & $1$ & $1$ & $1$ & $1$ & $1$\tabularnewline
\hline
1 & $\frac{1}{4}$ & $e^{-\frac{2\pi i}{3}}$ & $e^{-\frac{2\pi i}{3}}$ & $e^{-\frac{2\pi i}{3}}$ & $e^{-\frac{2\pi i}{3}}$ & $e^{-\frac{2\pi i}{3}}$ & $e^{-\frac{2\pi i}{3}}$\tabularnewline
\hline
2 & $\frac{1}{4}$ & $e^{\frac{2\pi i}{3}}$ & $e^{\frac{2\pi i}{3}}$ & $e^{\frac{2\pi i}{3}}$ & $e^{\frac{2\pi i}{3}}$ & $e^{\frac{2\pi i}{3}}$ & $e^{\frac{2\pi i}{3}}$\tabularnewline
\hline
3 & $\frac{3}{4}$ & $0$ & $0$ & $0$ & $0$ & $0$ & $0$\tabularnewline
\hline
4 & $\frac{3}{2}$ & $0$ & $0$ & $0$ & $0$ & $0$ & $0$\tabularnewline
\hline
5 & $\frac{3}{2}$ & $0$ & $0$ & $0$ & $0$ & $0$ & $0$\tabularnewline
\hline
6 & $1$ & $e^{-\frac{\pi i}{9}}$ & $e^{\frac{5\pi i}{9}}$ & $e^{-\frac{7\pi i}{9}}$ & $e^{\frac{2\pi i}{9}}$ & $e^{-\frac{4\pi i}{9}}$ & $e^{\frac{8\pi i}{9}}$\tabularnewline
\hline
7 & $1$ & $e^{\frac{5\pi i}{9}}$ & $\frac{}{}$$ $$e^{-\frac{7\pi i}{9}}$ & $e^{-\frac{\pi i}{9}}$ & $e^{\frac{8\pi i}{9}}$ & $e^{\frac{2\pi i}{9}}$ & $e^{-\frac{4\pi i}{9}}$\tabularnewline
\hline
8 & $1$ & $e^{-\frac{7\pi i}{9}}$ & $e^{-\frac{\pi i}{9}}$ & $e^{\frac{5\pi i}{9}}$ & $e^{-\frac{4\pi i}{9}}$ & $e^{\frac{8\pi i}{9}}$ & $e^{\frac{2\pi i}{9}}$\tabularnewline
\hline
9 & $1$ & $e^{\frac{2\pi i}{9}}$ & $e^{\frac{8\pi i}{9}}$ & $e^{-\frac{4\pi i}{9}}$ & $e^{-\frac{4\pi i}{9}}$ & $e^{\frac{8\pi i}{9}}$ & $e^{\frac{2\pi i}{9}}$\tabularnewline
\hline
10 & $1$ & $e^{-\frac{4\pi i}{9}}$ & $e^{\frac{2\pi i}{9}}$ & $e^{\frac{8\pi i}{9}}$ & $e^{-\frac{8\pi i}{9}}$ & $e^{\frac{2\pi i}{9}}$ & $e^{-\frac{4\pi i}{9}}$\tabularnewline
\hline
11 & $1$ & $e^{\frac{8\pi i}{9}}$ & $e^{-\frac{4\pi i}{9}}$ & $e^{\frac{2\pi i}{9}}$ & $e^{\frac{2\pi i}{9}}$ & $e^{-\frac{4\pi i}{9}}$ & $e^{\frac{8\pi i}{9}}$\tabularnewline
\hline
12 & $1$ & $e^{\frac{\pi i}{9}}$ & $e^{-\frac{5\pi i}{9}}$ & $e^{\frac{7\pi i}{9}}$ & $e^{-\frac{2\pi i}{9}}$ & $e^{\frac{4\pi i}{9}}$ & $e^{-\frac{8\pi i}{9}}$\tabularnewline
\hline
13 & $1$ & $e^{-\frac{5\pi i}{9}}$ & $e^{\frac{7\pi i}{9}}$ & $e^{\frac{\pi i}{9}}$ & $e^{-\frac{8\pi i}{9}}$ & $e^{-\frac{2\pi i}{9}}$ & $e^{\frac{4\pi i}{9}}$\tabularnewline
\hline
14 & $1$ & $e^{\frac{7\pi i}{9}}$ & $e^{\frac{\pi i}{9}}$ & $e^{-\frac{5\pi i}{9}}$ & $e^{\frac{4\pi i}{9}}$ & $e^{-\frac{8\pi i}{9}}$ & $e^{-\frac{2\pi i}{9}}$\tabularnewline
\hline
15 & $1$ & $e^{-\frac{2\pi i}{9}}$ & $e^{-\frac{8\pi i}{9}}$ & $e^{\frac{4\pi i}{9}}$ & $e^{\frac{4\pi i}{9}}$ & $e^{-\frac{8\pi i}{9}}$ & $e^{-\frac{2\pi i}{9}}$\tabularnewline
\hline
16 & $1$ & $e^{\frac{4\pi i}{9}}$ & $e^{-\frac{2\pi i}{9}}$ & $e^{-\frac{8\pi i}{9}}$ & $e^{-\frac{8\pi i}{9}}$ & $e^{-\frac{2\pi i}{9}}$ & $e^{\frac{4\pi i}{9}}$\tabularnewline
\hline
17 & $1$ & $e^{-\frac{8\pi i}{9}}$ & $e^{\frac{4\pi i}{9}}$ & $e^{-\frac{2\pi i}{9}}$ & $e^{-\frac{2\pi i}{9}}$ & $e^{\frac{4\pi i}{9}}$ & $e^{-\frac{8\pi i}{9}}$\tabularnewline
\hline
18 & $\frac{1}{2}$ & $e^{-\frac{2\pi i}{3}}$ & $e^{-\frac{2\pi i}{3}}$ & $e^{-\frac{2\pi i}{3}}$ & $e^{\frac{\pi i}{3}}$ & $e^{\frac{\pi i}{3}}$ & $e^{\frac{\pi i}{3}}$\tabularnewline
\hline
19 & $\frac{1}{2}$ & $1$ & $1$ & $1$ & $-1$ & $-1$ & $-1$\tabularnewline
\hline
20 & $\frac{1}{2}$ & $e^{\frac{2\pi i}{3}}$ & $e^{\frac{2\pi i}{3}}$ & $e^{\frac{2\pi i}{3}}$ & $e^{-\frac{\pi i}{3}}$ & $e^{-\frac{\pi i}{3}}$ & $e^{-\frac{\pi i}{3}}$\tabularnewline
\hline
\end{tabular}

\begin{tabular}{|c|c|c|c|c|c|c|}
\hline
$\sqrt{18}S_{i,j}$ & 12 & 13 & 14 & 15 & 16 & 17\tabularnewline
\hline
\hline
0 & $1$ & $1$ & $1$ & $1$ & $1$ & 1\tabularnewline
\hline
1 & $e^{\frac{2\pi i}{3}}$ & $e^{\frac{2\pi i}{3}}$ & $e^{\frac{2\pi i}{3}}$ & $e^{\frac{2\pi i}{3}}$ & $e^{\frac{2\pi i}{3}}$ & $e^{\frac{2\pi i}{3}}$\tabularnewline
\hline
2 & $e^{-\frac{2\pi i}{3}}$ & $e^{-\frac{2\pi i}{3}}$ & $e^{-\frac{2\pi i}{3}}$ & $e^{-\frac{2\pi i}{3}}$ & $e^{-\frac{2\pi i}{3}}$ & $e^{-\frac{2\pi i}{3}}$\tabularnewline
\hline
3 & $0$ & $0$ & $0$ & $0$ & $0$ & $0$\tabularnewline
\hline
4 & $0$ & $0$ & $0$ & $0$ & $0$ & $0$\tabularnewline
\hline
5 & $0$ & $0$ & $0$ & $0$ & $0$ & $0$\tabularnewline
\hline
6 & $e^{\frac{\pi i}{9}}$ & $e^{-\frac{5\pi i}{9}}$ & $e^{\frac{7\pi i}{9}}$ & $e^{-\frac{2\pi i}{9}}$ & $e^{\frac{4\pi i}{9}}$ & $e^{-\frac{8\pi i}{9}}$\tabularnewline
\hline
7 & $e^{-\frac{5\pi i}{9}}$ & $e^{\frac{7\pi i}{9}}$ & $e^{\frac{\pi i}{9}}$ & $e^{-\frac{8\pi i}{9}}$ & $e^{-\frac{2\pi i}{9}}$ & $e^{\frac{4\pi i}{9}}$\tabularnewline
\hline
8 & $e^{\frac{7\pi i}{9}}$ & $e^{\frac{\pi i}{9}}$ & $e^{-\frac{5\pi i}{9}}$ & $e^{\frac{4\pi i}{9}}$ & $e^{-\frac{8\pi i}{9}}$ & $e^{-\frac{2\pi i}{9}}$\tabularnewline
\hline
9 & $e^{-\frac{2\pi i}{9}}$ & $e^{-\frac{8\pi i}{9}}$ & $e^{\frac{4\pi i}{9}}$ & $e^{\frac{4\pi i}{9}}$ & $e^{-\frac{8\pi i}{9}}$ & $e^{-\frac{2\pi i}{9}}$\tabularnewline
\hline
10 & $e^{\frac{4\pi i}{9}}$ & $e^{-\frac{2\pi i}{9}}$ & $e^{-\frac{8\pi i}{9}}$ & $e^{-\frac{8\pi i}{9}}$ & $e^{-\frac{2\pi i}{9}}$ & $e^{\frac{4\pi i}{9}}$\tabularnewline
\hline
11 & $e^{-\frac{8\pi i}{9}}$ & $e^{\frac{4\pi i}{9}}$ & $e^{-\frac{2\pi i}{9}}$ & $e^{-\frac{2\pi i}{9}}$ & $e^{\frac{4\pi i}{9}}$ & $e^{-\frac{8\pi i}{9}}$\tabularnewline
\hline
12 & $e^{-\frac{\pi i}{9}}$ & $e^{\frac{5\pi i}{9}}$ & $e^{-\frac{7\pi i}{9}}$ & $e^{\frac{2\pi i}{9}}$ & $e^{-\frac{4\pi i}{9}}$ & $e^{\frac{8\pi i}{9}}$\tabularnewline
\hline
13 & $e^{\frac{5\pi i}{9}}$ & $\frac{}{}$$ $$e^{-\frac{7\pi i}{9}}$ & $e^{-\frac{\pi i}{9}}$ & $e^{\frac{8\pi i}{9}}$ & $e^{\frac{2\pi i}{9}}$ & $e^{-\frac{4\pi i}{9}}$\tabularnewline
\hline
14 & $e^{-\frac{7\pi i}{9}}$ & $e^{-\frac{\pi i}{9}}$ & $e^{\frac{5\pi i}{9}}$ & $e^{-\frac{4\pi i}{9}}$ & $e^{\frac{8\pi i}{9}}$ & $e^{\frac{2\pi i}{9}}$\tabularnewline
\hline
15 & $e^{\frac{2\pi i}{9}}$ & $e^{\frac{8\pi i}{9}}$ & $e^{-\frac{4\pi i}{9}}$ & $e^{-\frac{4\pi i}{9}}$ & $e^{\frac{8\pi i}{9}}$ & $e^{\frac{2\pi i}{9}}$\tabularnewline
\hline
16 & $e^{-\frac{4\pi i}{9}}$ & $e^{\frac{2\pi i}{9}}$ & $e^{\frac{8\pi i}{9}}$ & $e^{-\frac{8\pi i}{9}}$ & $e^{\frac{2\pi i}{9}}$ & $e^{-\frac{4\pi i}{9}}$\tabularnewline
\hline
17 & $e^{\frac{8\pi i}{9}}$ & $e^{-\frac{4\pi i}{9}}$ & $e^{\frac{2\pi i}{9}}$ & $e^{\frac{2\pi i}{9}}$ & $e^{-\frac{4\pi i}{9}}$ & $e^{\frac{8\pi i}{9}}$\tabularnewline
\hline
18 & $e^{\frac{2\pi i}{3}}$ & $e^{\frac{2\pi i}{3}}$ & $e^{\frac{2\pi i}{3}}$ & $e^{-\frac{\pi i}{3}}$ & $e^{-\frac{\pi i}{3}}$ & $e^{-\frac{\pi i}{3}}$\tabularnewline
\hline
19 & $1$ & $1$ & $1$ & $-1$ & $-1$ & $-1$\tabularnewline
\hline
20 & $e^{-\frac{2\pi i}{3}}$ & $e^{-\frac{2\pi i}{3}}$ & $e^{-\frac{2\pi i}{3}}$ & $e^{\frac{\pi i}{3}}$ & $e^{\frac{\pi i}{3}}$ & $e^{\frac{\pi i}{3}}$\tabularnewline
\hline
\end{tabular}

\end{document}